\documentclass[11pt, letterpaper]{amsart}
\usepackage{fullpage,amsfonts,amsmath,amssymb,amsthm,tikz}
\usetikzlibrary{matrix}
\usepackage{bm}
\usepackage[vcentermath]{youngtab}
\usepackage{young}
\usepackage[all,cmtip]{xy}
\usepackage{hyperref}
\usepackage{enumitem}
\newtheorem{theorem}{Theorem}[section]
\newtheorem{corollary}[theorem]{Corollary}
\newtheorem{lemma}[theorem]{Lemma}
\newtheorem{proposition}[theorem]{Proposition}

\theoremstyle{definition}
\newtheorem{definition}{Definition}[section]
\newtheorem{example}{Example}[section]
\newtheorem*{remark}{Remark}

\begin{document}
\title{The Hopf monoid on nonnesting supercharacters of pattern groups}
\author{Scott Andrews}
\address{Department of Mathematics \\ University of Colorado Boulder \\ Boulder, CO 80309}
\email{scott.andrews@colorado.edu}
\urladdr{http://euclid.colorado.edu/~sda/}

\begin{abstract} We construct supercharacter theories for a collection of unipotent matrix groups and produce a Hopf monoid from the supercharacters. These supercharacter theories are coarser than those defined by Diaconis--Isaacs for algebra groups and have supercharacters and superclasses indexed by nonnesting labeled set partitions. We compute the supercharacter tables and describe the product and coproduct of the Hopf monoid combinatorially. We also show that this Hopf monoid is free.
\end{abstract}
\maketitle

\section{Introduction}
Let $UT_n(\mathbb{F}_q)$ be the group of unipotent upper triangular $n \times n$ matrices with entries in the finite field with $q$ elements. Indexing the irreducible representations and conjugacy classes of $UT_n(\mathbb{F}_q)$ is a `wild' problem (see \cite{MR1056208}). It is not even known whether for a fixed $n$ the number of irreducible representations of $UT_n(\mathbb{F}_q)$ is polynomial in $q$. This conjecture of Higman (see \cite{MR0113948}) has been verified for $n \leq 13$ by Vera-L\'opez--Arregi in \cite{MR1994321}.

\bigbreak

Recent attempts towards understanding the representation theory of $UT_n(\mathbb{F}_q)$ have utilized `supercharacter theories.' Loosely speaking, a supercharacter theory of a group consists of a partition of the group into unions of conjugacy classes, along with a collection of characters of the group, such that the characters and the blocks of the partition behave like irreducible characters and conjugacy classes. Supercharacters are a generalization of the `basic characters' of $UT_n(\mathbb{F}_q)$ studied by Andr\'e in \cite{MR1839342,MR1338979,MR1896026} and Yan in \cite{yan}. The basic characters are indexed by labeled set partitions, and the values of these characters on elements of $UT_n(\mathbb{F}_q)$ have been computed in terms of the corresponding labeled set partitions. Furthermore, for a fixed $n$ the number of basic characters of $UT_n(\mathbb{F}_q)$ is polynomial in $q$.

\bigbreak

In \cite{MR2373317}, Diaconis--Isaacs define a supercharacter theory of an arbitrary finite group and generalize the construction of Yan to a large collection of subgroups of $UT_n(\mathbb{F}_q)$. A subgroup $U$ of $UT_n(\mathbb{F}_q)$ is called a \emph{pattern group} if $\{u-1 \mid u \in U\}$ is an $\mathbb{F}_q$-algebra with a basis consisting of elementary matrices (in particular, $UT_n(\mathbb{F}_q)$ is a pattern group). The construction of Diaconis--Isaacs produces a supercharacter theory for any pattern group. Diaconis--Thiem provide a supercharacter formula for an arbitrary pattern group in \cite{MR2491890}; their formula, however, is not in terms of a nice combinatorial indexing set.

\bigbreak

Pattern subgroups of $UT_n(\mathbb{F}_q)$ are naturally in bijection with posets on $\{1,2,...,n\}$ that can be extended to the usual linear order. In \cite{MR2855781}, Halasi--P\'alfy demonstrate that there exists a poset such that the number of irreducible characters of the corresponding pattern group is not polynomial in $q$. As we note in Section~\ref{secnnp}, the number of supercharacters of the pattern group corresponding to this poset is also not polynomial in $q$. This helps to explain why there is no known indexing set for supercharacters of pattern groups and no simple combinatorial formula for the supercharacter table.

\bigbreak

In this paper we construct a new, coarser supercharacter theory for an arbitrary pattern group. The supercharacters and superclasses have a nice combinatorial indexing set, and the supercharacter table is described in terms of this indexing set. For a fixed poset, the number of supercharacters of the corresponding pattern group is polynomial in $q$.

\bigbreak

Aguiar et al. construct a Hopf algebra from the supercharacters of $UT_n(\mathbb{F}_q)$ in \cite{MR2880223} and show that this Hopf algebra is isomorphic to the Hopf algebra of symmetric functions in noncommuting variables. Aguiar--Bergeron--Thiem show that this Hopf algebra comes from a Hopf monoid in \cite{MR3117506}. We produce a Hopf monoid from our supercharacters in an analogous manner and describe the product and coproduct combinatorially. In particular, we show that the resulting Hopf monoid is free.

\bigbreak

We begin by presenting background material on labeled set partitions, supercharacter theories, and Hopf monoids in Sections~\ref{seclsp}, \ref{secsctpg}, and \ref{secvshm}. In Section~\ref{secnnutn} we construct a supercharacter theory of $UT_n(\mathbb{F}_q)$ with supercharacters and superclasses indexed by `nonnesting $\mathbb{F}_q$-set partitions' (as in Section~\ref{seclsp}). We generalize this construction to arbitrary pattern groups in Section~\ref{secnnp} and in Section~\ref{sechmpg} construct a Hopf monoid from these supercharacters. Finally, in Section~\ref{seccomb} we provide combinatorial descriptions of the product and coproduct and show that this Hopf monoid is free.

\section{Labeled set partitions}\label{seclsp}

Let $[n]=\{1,2,...,n\}$. A \emph{set partition} of $[n]$ is a collection $\{S_i\}$ of nonempty subsets of $[n]$ such that $[n]$ is the disjoint union of the $S_i$. For our purposes, it will be useful to represent set partitions as \emph{arc diagrams}. Given a set partition $\nu = \{S_i\}$ of $[n]$, we consider the set of edges
\[
        \{k \frown l \mid \text{there exists } i \text{ with } k,l \in S_i,\text{ and if } k<j<l, \text{ then } j \notin S_i\},
\]
from which we construct an arc diagram. For example, consider the set partition $\{\{1,3,7\},\{2\},\{4,5,8\}\}$. Corresponding to this set partition is the arc diagram
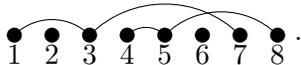
\[
        \begin{tikzpicture}[baseline={([yshift=-.5ex]current bounding box.center)}]
	\fill (0,0) circle (.1) node[below]{1};
    \fill (.5,0) circle (.1) node[below]{2};
    \fill (1,0) circle (.1) node[below]{3};
    \fill (1.5,0) circle (.1) node[below]{4};
    \fill (2,0) circle (.1) node[below]{5};
    \fill (2.5,0) circle (.1) node[below]{6};
    \fill (3,0) circle (.1) node[below]{7};
    \fill (3.5,0) circle (.1) node[below]{8};
    \draw (0,0) to [out=45, in=135] (1,0);
    \draw (1,0) to [out=45, in=135] (3,0);
    \draw (1.5,0) to [out=45, in=135] (2,0);
    \draw (2,0) to [out=45, in=135] (3.5,0);
    \end{tikzpicture}.
\]
We will usually omit the labels from the nodes of arc diagrams. From now on we will conflate the notions of set partitions, arc diagrams, and edge sets of arc diagrams. For instance, we will write
\begin{align*}
\nu & = \{\{1,3,7\},\{2\},\{4,5,8\}\}, \\
\nu & = \{1 \frown 3, 3 \frown 7, 4 \frown 5, 5 \frown 8\},\text{ and} \\
\nu & = \begin{tikzpicture}
	\fill (0,-1) circle (.1);
    \fill (.5,-1) circle (.1);
    \fill (1,-1) circle (.1);
    \fill (1.5,-1) circle (.1);
    \fill (2,-1) circle (.1);
    \fill (2.5,-1) circle (.1);
    \fill (3,-1) circle (.1);
    \fill (3.5,-1) circle (.1);
    \draw (0,-1) to [out=45, in=135] (1,-1);
    \draw (1,-1) to [out=45, in=135] (3,-1);
    \draw (1.5,-1) to [out=45, in=135] (2,-1);
    \draw (2,-1) to [out=45, in=135] (3.5,-1);
    \end{tikzpicture},
\end{align*}
all referring to the same set partition. Note that the set partition can be recovered from the arc diagram by considering the connected components. An arc diagram $\nu$ will correspond to a set partition if and only if for all $i \frown j \in\nu$ and $i <k<j$, $i \frown k,k\frown j \notin \nu$.

\bigbreak

One advantage to representing set partitions as arc diagrams is that we can introduce labels to the arcs. If $\mathbb{F}_q$ is the finite field with $q$ elements, an $\mathbb{F}_q$\emph{-set partition} of $[n]$ is a labeled arc diagram of the form
\[
        \nu = \{i \overset{a}{\frown}j \mid a \in \mathbb{F}_q^\times \text{ and if }i <k<j, \text{ then } i \overset{b}{\frown} k,k\overset{c}{\frown} j \notin \nu\}.
\]
If $\eta$ is a set partition, we will say that $\eta$ is \emph{nonnesting} if there are no $i<j<k<l$ such that $i \frown l,j\frown k \in \eta$.

\begin{example} Let
\begin{align*}
        \eta &= \begin{tikzpicture}[baseline={([yshift=-.5ex]current bounding box.center)}]
	\fill (0,0) circle (.1) node[below] {$1$};
    \fill (.5,0) circle (.1) node[below] {$2$};
    \fill (1,0) circle (.1) node[below] {$3$};
    \fill (1.5,0) circle (.1) node[below] {$4$};
    \fill (2,0) circle (.1) node[below] {$5$};
    \fill (2.5,0) circle (.1) node[below] {$6$};
    \fill (3,0) circle (.1) node[below] {$7$};
    \fill (3.5,0) circle (.1) node[below] {$8$};
    \draw (0,0) to [out=45, in=135] (1,0);
    \draw (1,0) to [out=45, in=135] (3,0);
    \draw (0.5,0) to [out=45, in=135] (2,0);
    \draw (2,0) to [out=45, in=135] (3.5,0);
    \end{tikzpicture} \quad\text{and} \\
    \nu &= \begin{tikzpicture}[baseline={([yshift=-.5ex]current bounding box.center)}]
	\fill (0,0) circle (.1) node[below] {$1$};
    \fill (.5,0) circle (.1) node[below] {$2$};
    \fill (1,0) circle (.1) node[below] {$3$};
    \fill (1.5,0) circle (.1) node[below] {$4$};
    \fill (2,0) circle (.1) node[below] {$5$};
    \fill (2.5,0) circle (.1) node[below] {$6$};
    \fill (3,0) circle (.1) node[below] {$7$};
    \fill (3.5,0) circle (.1) node[below] {$8$};
    \draw (0,0) to [out=45, in=135] (1,0);
    \draw (1,0) to [out=45, in=135] (3,0);
    \draw (1.5,0) to [out=45, in=135] (2,0);
    \draw (2,0) to [out=45, in=135] (3.5,0);
    \end{tikzpicture};
\end{align*}
then $\eta$ is a nonnesting set partition but $\nu$ is not, as $3 \frown 7, 4 \frown 5$ is a nesting.
\end{example}

If the underlying set partition of an $\mathbb{F}_q$-set partition is nonnesting, we will call that $\mathbb{F}_q$-set partition nonnesting. Following the notation from \cite{MR2880659}, define
\begin{align*}
        \Pi(n,q) & = \{\mathbb{F}_q\text{-set partitions of } [n]\} \quad \text{and} \\
        \text{NN}(n,q) & = \{\text{nonnesting }\mathbb{F}_q\text{-set partitions of } [n]\}.
\end{align*}

Let $I$ be a finite set and let $\mathcal{P}$ be a poset on $I$. A $\mathcal{P}$\emph{-set partition} of $I$ is an arc diagram $\nu$ such that
\begin{enumerate}
\item if $i \frown j \in \nu$ then $i \prec_\mathcal{P} j$; and
\item for all $i \frown j \in \nu$ and $i \prec_\mathcal{P} k \prec_\mathcal{P} j$, we have $i \frown k,k \frown j \notin \nu$.
\end{enumerate}
If $\mathcal{P}$ is the usual linear order on $[n]$ then a $\mathcal{P}$-set partition is the same as a set partition.

\begin{example} Let $I = \{1,2,3,4,5,6\}$ and let $\mathcal{P}$ be the poset on $I$ with Hasse diagram
\[
    \begin{tikzpicture}[baseline={([yshift=-.5ex]current bounding box.center)}]
    \node (e) at (0,2) {$5$};
    \node (a) at (-3,0) {$1$};
    \node (b) at (-1,0) {$2$};
    \node (c) at (1,0) {$3$};
    \node (d) at (3,0) {$4$};
    \node (f) at (0,-2) {$6$};
    \draw (e) -- (a);
    \draw (e) -- (b) -- (f) -- (c) -- (e) -- (d);
\end{tikzpicture}.
\]
We have that
\[
    \nu = \begin{tikzpicture}[baseline={([yshift=-.5ex]current bounding box.center)}]
	\fill (0,0) circle (.1) node[below]{6};
    \fill (.5,0) circle (.1) node[below]{1};
    \fill (1,0) circle (.1) node[below]{2};
    \fill (1.5,0) circle (.1) node[below]{3};
    \fill (2,0) circle (.1) node[below]{4};
    \fill (2.5,0) circle (.1) node[below]{5};
    \draw (0,0) to [out=45, in=135] (1,0);
    \draw (0,0) to [out=45, in=135] (1.5,0);
    \draw (.5,0) to [out=45, in=135] (2.5,0);
    \draw (2,0) to [out=45, in=135] (2.5,0);
    \end{tikzpicture}
\]
is a valid $\mathcal{P}$-set partition of $I$. Note that in representing the arc diagram as above we are implicitly viewing $\mathcal{P}$ as a subset of the linear order $6 \prec 1\prec 2 \prec 3\prec 4 \prec 5$; we could just as well consider $\mathcal{P}$ as a subset of the linear order $1 \prec 4\prec 6 \prec 3\prec 2 \prec 5$ and represent the arc diagram as
\[
        \nu = \begin{tikzpicture}[baseline={([yshift=-.5ex]current bounding box.center)}]
	\fill (0,0) circle (.1) node[below]{1};
    \fill (.5,0) circle (.1) node[below]{4};
    \fill (1,0) circle (.1) node[below]{6};
    \fill (1.5,0) circle (.1) node[below]{3};
    \fill (2,0) circle (.1) node[below]{2};
    \fill (2.5,0) circle (.1) node[below]{5};
    \draw (0,0) to [out=45, in=135] (2.5,0);
    \draw (.5,0) to [out=45, in=135] (2.5,0);
    \draw (1,0) to [out=45, in=135] (1.5,0);
    \draw (1,0) to [out=45, in=135] (2,0);
    \end{tikzpicture}.
\]
Although the representations look different, these arc diagrams have the same edge set.
\end{example}

Let $\nu$ be a $\mathcal{P}$-set partition; we say that $\nu$ is \emph{nonnesting} if there are no $i \frown j,k \frown l \in \nu$ with $i \prec_\mathcal{P} k\prec_\mathcal{P} l \prec_\mathcal{P} j$. Note that if $\mathcal{P}$ is the usual linear order on $[n]$, the nonnesting $\mathcal{P}$-set partitions are the same as the nonnesting set partitions.

\bigbreak

Once again we want to label the arcs of our arc diagrams. An $(\mathbb{F}_q,\mathcal{P})$\emph{-set partition} is a $\mathcal{P}$-set partition with arcs labeled by elements of $\mathbb{F}_q^\times$. If the underlying $\mathcal{P}$-set partition is nonnesting we will also refer to the $(\mathbb{F}_q,\mathcal{P})$-set partition as nonnesting.

\bigbreak

From now on, let
\begin{align*}
        \Pi(\mathcal{P},q) & = \{(\mathbb{F}_q,\mathcal{P})\text{-set partitions}\} \quad \text{and} \\
         \text{NN}(\mathcal{P},q) & = \{\text{nonnesting }(\mathbb{F}_q,\mathcal{P})\text{-set partitions}\}.
\end{align*}

\begin{remark} The notion of a $\mathcal{P}$-set partition of $I$ is not the same as that of a `poset partition' (see, for instance, \cite[Section 13.1.4]{MR2724388}).
\end{remark}

\section{Supercharacter theories of pattern groups}\label{secsctpg}

In \cite{MR2373317}, Diaconis--Isaacs define a supercharacter theory of an arbitrary finite group $G$. In this section we review this definition and present supercharacter theories of a large collection of unipotent groups.

\begin{definition} Let $G$ be a finite group, $\mathcal{K}$ be a partition of $G$ into unions of conjugacy classes, and $\mathcal{X}$ be a set of characters of $G$. We say that the pair $(\mathcal{K},\mathcal{X})$ is a \emph{supercharacter theory} of $G$ if
\begin{enumerate}[leftmargin=.52in]
\item[(SCT1)] $|\mathcal{X}|=|\mathcal{K}|$,
\item[(SCT2)] the characters $\chi\in \mathcal{X}$ are constant on the members of $\mathcal{K}$, and
\item[(SCT3)] each irreducible character of $G$ is a constituent of exactly one character in $\mathcal{X}$.
\end{enumerate}
The characters $\chi \in \mathcal{X}$ are called \emph{supercharacters} and the blocks $K \in \mathcal{K}$ are called \emph{superclasses}.
\end{definition}

Let $UT_n(\mathbb{F}_q)$ denote the group of upper-triangular matrices with entries in the finite field with $q$ elements and ones on the diagonal. The initial motivation behind supercharacter theories was the study of `basic characters' of $UT_n(\mathbb{F}_q)$ by Andr\'e in \cite{MR1839342,MR1338979,MR1896026}. Andr\'e's construction was reframed in terms of group actions by Yan in \cite{yan} before being generalized to a large collection of subgroups of $UT_n(\mathbb{F}_q)$ by Diaconis--Isaacs in \cite{MR2373317}.

\bigbreak

Let $\frak{g}$ be a nilpotent $\mathbb{F}_q$-algebra. The \emph{algebra group} associated to $\frak{g}$ is the set
\[
        G = \{1+x \mid x \in \frak{g}\}
\]
with multiplication $(1+x)(1+y) = 1+x+y+xy$. We will write $G = 1+\frak{g}$ to indicate that $G$ is the algebra group associated to $\frak{g}$. For example, let $\frak{ut}_n(\mathbb{F}_q)$ be the algebra of strictly upper-triangular matrices with entries in $\mathbb{F}_q$; then $UT_n(\mathbb{F}_q)$ is the algebra group associated to $\frak{ut}_n(\mathbb{F}_q)$.

\bigbreak

For a finite set $I$, define an $\mathbb{F}_q$-algebra
\[
        M(I) = \{a = (a_{ij}) \mid i,j \in I \text{ and } a_{ij} \in \mathbb{F}_q\}
\]
with addition and scalar multiplication defined pointwise and multiplication
\[
        (ab)_{ij} = \sum_{k \in I}a_{ik}b_{kj}.
\]
If a total order is chosen on $I$, then $M(I)$ is canonically isomorphic to the algebra of matrices over $\mathbb{F}_q$ with entries indexed by the elements of $I$. If $\mathcal{P}$ is a poset on $I$, we define the \emph{pattern algebra} associated to $\mathcal{P}$ by
\[
        \frak{u}_\mathcal{P} = \{x \in M(I) \mid x_{ij}=0 \text{ unless } i\prec_\mathcal{P} j \}.
\]
The \emph{pattern group} associated to $\mathcal{P}$ is defined to be $U_\mathcal{P} = 1+\frak{u}_\mathcal{P}$; as $\frak{u}_\mathcal{P}$ is nilpotent, $U_\mathcal{P}$ is an algebra group. For more on algebra groups and pattern groups, see \cite{MR2373317,MR2491890,MR1358482}.

\bigbreak

In \cite{MR2373317}, Diaconis--Isaacs construct a supercharacter theory of an arbitrary algebra group $G = 1+\frak{g}$. Define a bijection
\begin{align*}
        f:G &\to \frak{g} \\
        1+x & \mapsto x.
\end{align*}
The group $G$ acts by left and right multiplication on $\frak{g}$; for $g\in G$, define
\[
        K_g = \{h \in G \mid f(h) \in Gf(g)G\}.
\]
The partition $\{K_g \mid g \in G\}$ is a partition of $G$ into unions of conjugacy classes.

\bigbreak

There are contragedient actions of $G$ on the dual $\frak{g}^*$ defined by
\[
        (g\lambda)(x) = \lambda(g^{-1}x) \quad \text{and} \quad (\lambda g)(x) = \lambda(xg^{-1}),
\]
where $g \in G$, $x \in \frak{g}$, and $\lambda \in \frak{g}^*$. Let $\theta : \mathbb{F}_q^+ \to \mathbb{C}^\times$ be a nontrivial homomorphism, and for $\lambda \in \frak{g}^*$ define
\[
        \chi_\lambda = \sum_{\mu \in G \lambda G} \theta \circ \mu \circ f.
\]
\begin{remark} Our definition of $\chi_\lambda$ differs from that in \cite{MR2373317} by a constant multiple.
\end{remark}

\begin{theorem}[Diaconis--Isaacs, \cite{MR2373317}]\label{alggpsct} The functions $\chi_\lambda$ are characters of $G$, and the pair
\[
        (\{K_g \mid g \in G\},\{\chi_\lambda \mid \lambda \in \frak{g}^*\})
\]
defines a supercharacter theory of $G$.
\end{theorem}

In the case that $G = UT_n(\mathbb{F}_q)$, there is a nice description of the supercharacters and superclasses in terms of $\mathbb{F}_q$-set partitions. Propositions~\ref{lemtypeascl} and \ref{lemtypeaschdim}, as well as Corollary~\ref{cortypeaschscl} and Theorem~\ref{theoremtypeaschtable}, were all initially developed by Andr\'e (for instance, in \cite{MR1839342}), but not phrased in terms of $\mathbb{F}_q$-set partitions.

\bigbreak

Given an $\mathbb{F}_q$-set partition $\eta$, define $x_\eta \in \frak{ut}_n(\mathbb{F}_q)$ and $\lambda_\eta \in \frak{ut}_n(\mathbb{F}_q)^*$ by
\[
        x_\eta = \sum_{i \overset{a}{\frown}j \in \eta} ae_{ij} \quad \text{and}\quad\lambda_\eta(x) = \sum_{i \overset{a}{\frown}j \in \eta} ax_{ij},
\]
where $x \in \frak{g}$.

\begin{proposition}[{\cite[Section 2.3]{MR2592079}}]\label{lemtypeascl} We have the following.
\begin{enumerate}
\item The set
\[
        \{x_\eta \mid \eta \text{ is an }\mathbb{F}_q\text{-set partition}\}
\]
is a set of orbit representatives for the action of $UT_n(\mathbb{F}_q) \times UT_n(\mathbb{F}_q)$ on $\frak{ut}_n(\mathbb{F}_q)$.
\item The set
\[
    \{\lambda_\eta \mid \eta \text{ is an }\mathbb{F}_q\text{-set partition}\}
\]
is a set of orbit representatives for the action of $UT_n(\mathbb{F}_q) \times UT_n(\mathbb{F}_q)$ on $\frak{ut}_n(\mathbb{F}_q)^*$.
\end{enumerate}
\end{proposition}

For an $\mathbb{F}_q$-set partition $\eta$, define $g_\eta \in UT_n(\mathbb{F}_q)$ by $g_\eta = 1+x_\eta$ and let $\chi_\eta = \chi_{\lambda_\eta}$.

\begin{corollary}[{\cite[Section 2.3]{MR2592079}}]\label{cortypeaschscl} The sets
\[
        \mathcal{X} = \{\chi_\eta \mid \eta \text{ is an }\mathbb{F}_q\text{-set partition}\} \quad \text{and} \quad \mathcal{K} = \{K_{g_\nu}\mid\nu\text{ is an }\mathbb{F}_q\text{-set partition}\}
\]
are the supercharacters and superclasses of the supercharacter theory of $UT_n(\mathbb{F}_q)$ defined in Theorem~\ref{alggpsct}.
\end{corollary}

We can calculate the dimensions of the supercharacters in terms of the corresponding $\mathbb{F}_q$-set partitions.

\begin{proposition}\label{lemtypeaschdim} Let $\eta$ be an $\mathbb{F}_q$-set partition; then
\begin{align*}
        \chi_\eta(1) &= q^{2(\sum_{i \overset{a}{\frown}j \in \eta} j-i-1)-|C(\eta)|},
\end{align*}
where $C(\eta) = \{(i \overset{a}{\frown} k, j \overset{b}{\frown} l) \mid i \overset{a}{\frown} k, j \overset{b}{\frown} l \in \eta \text{ and } i<j<k<l\}$ is the \emph{crossing set} of $\eta$ (as in \cite[Section 2.1]{MR2592079}).
\end{proposition}

We mention that the supercharacters that we are considering differ from those in \cite{MR2592079} by a constant multiple, and as such the dimension $\chi_\eta(1)$ in Proposition~\ref{lemtypeaschdim} differs from that in \cite[Equation~2.2]{MR2592079}. Finally we calculate the values of the supercharacters on the superclasses.

\begin{theorem}[{\cite[Equation 2.1]{MR2592079}}]\label{theoremtypeaschtable} Let $\eta$ and $\nu$ be $\mathbb{F}_q$-set partitions; then we have
\[
        \chi_{\eta}(g_\nu) = \left\{\begin{array}{ll}
        \frac{\chi_{\eta}(1)}{q^{\textup{nst}_\nu^\eta}}
        \displaystyle\prod_{\substack{i\overset{a}{\frown}j \in \eta \\ i\overset{b}{\frown}j \in \nu}} \theta(ab) & \quad  \begin{array}{l}\text{if for }i\overset{a}{\frown}j \in \eta \text{ and } i<k<j, \\ i\overset{b}{\frown}k, k \overset{b}{\frown} j \notin \nu, \end{array} \\
        0 & \quad \text{otherwise,}\end{array}\right.
\]
where $\textup{nst}_\nu^\eta = |\{i<j<k<l \mid j\overset{a}{\frown} k \in \nu,i\overset{b}{\frown} l \in \eta\}|$.
\end{theorem}

Although this formula appears identical to \cite[Equation 2.1]{MR2592079}, once again it differs by a constant multiple. This is due to the fact that the dimension of the character $\chi_\eta(1)$ appears in the formula and is different in our case. Note that while the algebra group supercharacter theory exists for all pattern groups, the above formula for the value of a supercharacter on a superclass only applies to $UT_n(\mathbb{F}_q)$. In Proposition~\ref{anotherprop}, we will show that by considering a coarser supercharacter theory of a pattern group we get a similar formula that applies to all pattern groups.

\section{Vector species and Hopf monoids}\label{secvshm}

The ideas of vector species and Hopf monoids are explored in great detail in \cite{MR2724388}. We only present the necessary definitions and results for our construction.

\subsection{Definitions}

Let $\mathbf{Set}^\times$ denote the category of finite sets with morphisms given by bijections and $\mathbf{Vec}$ denote the category of vector spaces over the field $\mathbb{F}$. A \emph{vector species} $\mathbf{p}$ is a functor
\[
        \mathbf{p}:\mathbf{Set}^\times \to \mathbf{Vec}.
\]
In other words, $\mathbf{p}$ is a collection of vector spaces $\mathbf{p}[I]$ indexed by finite sets $I$ along with linear maps
\[
        \mathbf{p}[\sigma]:\mathbf{p}[I] \to \mathbf{p}[J]
\]
for each bijection $\sigma: I \to J$. These maps must satisfy that
\[
        \mathbf{p}[\text{id}] = \text{id}
\]
and
\[
        \mathbf{p}(\sigma \circ \tau) = \mathbf{p}(\sigma) \circ \mathbf{p}(\tau).
\]

\bigbreak

If $\mathbf{p}[\emptyset] = \mathbb{F}$, we say that $\mathbf{p}$ is a \emph{connected species}. A \emph{Hopf monoid} consists of a connected vector species $\mathbf{p}$ along with linear maps
\begin{align*}
\mu_{S,T} &: \mathbf{p}[S]\otimes\mathbf{p}[T] \to \mathbf{p}[S\sqcup T]\quad\text{and} \\
\Delta_{S,T} &: \mathbf{p}[S\sqcup T] \to \mathbf{p}[S]\otimes\mathbf{p}[T]
\end{align*}
for each pair of disjoint sets $S$ and $T$. These maps must respect bijections in the sense that if $\sigma:S \to S'$ and $\tau:T \to T'$ are bijections, then the diagrams
\begin{equation}\label{functorial1}
    \begin{tikzpicture}
  \matrix (m) [matrix of math nodes,row sep=3em,column sep=4em,minimum width=2em]
  {
     \mathbf{p}[S]\otimes\mathbf{p}[T] & \mathbf{p}[S\sqcup T] \\
     \mathbf{p}[S']\otimes\mathbf{p}[T'] & \mathbf{p}[S'\sqcup T'] \\};
  \path[-stealth]
    (m-1-1) edge node [left] {$\mathbf{p}[\sigma] \otimes \mathbf{p}[\tau]$} (m-2-1)
            edge node [above] {$\mu_{S,T}$} (m-1-2)
    (m-2-1) edge node [below] {$\mu_{S',T'}$}
            (m-2-2)
    (m-1-2) edge node [right] {$\mathbf{p}[\sigma\sqcup\tau]$} (m-2-2);
    \end{tikzpicture}
\end{equation}
and
\begin{equation}\label{functorial2}
    \begin{tikzpicture}
  \matrix (m) [matrix of math nodes,row sep=3em,column sep=4em,minimum width=2em]
  {
     \mathbf{p}[S]\otimes\mathbf{p}[T] & \mathbf{p}[S\sqcup T] \\
     \mathbf{p}[S']\otimes\mathbf{p}[T'] & \mathbf{p}[S'\sqcup T'] \\};
  \path[-stealth]
    (m-1-1) edge node [left] {$\mathbf{p}[\sigma] \otimes \mathbf{p}[\tau]$} (m-2-1)
     (m-1-2)   edge node [above] {$\Delta_{S,T}$} (m-1-1)
    (m-2-2) edge node [below] {$\Delta_{S',T'}$}
            (m-2-1)
    (m-1-2) edge node [right] {$\mathbf{p}[\sigma\sqcup\tau]$} (m-2-2);
    \end{tikzpicture}
\end{equation}
commute. If $\mu$ denotes the collection of $\mu_{S,T}$ and $\Delta$ the collection of $\Delta_{S,T}$, where each collection varies over all pairs of disjoint finite sets $S$ and $T$, then $\mu$ is called the \emph{product} and $\Delta$ is called the \emph{coproduct}.

\bigbreak

These collections must be \emph{associative} and \emph{coassociative} in that the diagrams
\begin{equation}\label{associative}
\begin{tikzpicture}
  \matrix (m) [matrix of math nodes,row sep=3em,column sep=4em,minimum width=2em]
  {
     \mathbf{p}[R]\otimes\mathbf{p}[S]\otimes\mathbf{p}[T] & \mathbf{p}[R] \otimes \mathbf{p}[S\sqcup T] \\
     \mathbf{p}[R\sqcup S]\otimes\mathbf{p}[T] & \mathbf{p}[R\sqcup S\sqcup T] \\};
  \path[-stealth]
    (m-1-1) edge node [left] {$\mu_{R,S}\otimes\text{id}$} (m-2-1)
            edge node [above] {$\text{id}\otimes\mu_{S,T}$} (m-1-2)
    (m-2-1) edge node [below] {$\mu_{R\sqcup S,T}$}
            (m-2-2)
    (m-1-2) edge node [right] {$\mu_{R,S\sqcup T}$} (m-2-2);
    \end{tikzpicture}
\end{equation}
and
\begin{equation}\label{coassociative}
\begin{tikzpicture}
  \matrix (m) [matrix of math nodes,row sep=3em,column sep=4em,minimum width=2em]
  {
     \mathbf{p}[R]\otimes\mathbf{p}[S]\otimes\mathbf{p}[T] & \mathbf{p}[R] \otimes \mathbf{p}[S\sqcup T] \\
     \mathbf{p}[R\sqcup S]\otimes\mathbf{p}[T] & \mathbf{p}[R\sqcup S\sqcup T] \\};
  \path[-stealth]
    (m-2-1) edge node [left] {$\Delta_{R,S}\otimes\text{id}$} (m-1-1)
    (m-1-2) edge node [above] {$\text{id}\otimes\Delta_{S,T}$} (m-1-1)
    (m-2-2) edge node [below] {$\Delta_{R\sqcup S,T}$}
            (m-2-1)
    (m-2-2) edge node [right] {$\Delta_{R,S\sqcup T}$} (m-1-2);
    \end{tikzpicture}
\end{equation}
commute for all pairwise disjoint finite sets $R$, $S$, and $T$. They also must satisfy a compatibility property. Let $(S_1,S_2)$ and $(T_1,T_2)$ be pairs of disjoint finite sets such that $S_1\sqcup S_2 = T_1 \sqcup T_2 =I$. Let
\[
        A = S_1 \cap T_1, \quad B = S_1\cap T_2, \quad C = S_2 \cap T_1,\quad\text{and} \quad D = S_2 \cap T_2;
\]
then the diagram
\begin{equation}\label{hopfcompatible}
\begin{tikzpicture}[font=\small]
  \matrix (m) [matrix of math nodes,row sep=3em,column sep=4em,minimum width=2em]
  {
     \mathbf{p}[S_1]\otimes\mathbf{p}[S_2] & \mathbf{p}[I] & \mathbf{p}[T_1]\otimes\mathbf{p}[T_2] \\
     \mathbf{p}[A]\otimes\mathbf{p}[B]\otimes \mathbf{p}[C]\otimes\mathbf{p}[D] & {} & \mathbf{p}[A]\otimes\mathbf{p}[C]\otimes \mathbf{p}[B]\otimes\mathbf{p}[D] \\};
  \path[-stealth]
    (m-1-1) edge node [above] {$\mu_{S_1,S_2}$} (m-1-2)
    (m-1-2) edge node [above] {$\Delta_{T_1,T_2}$} (m-1-3)
    (m-2-1) edge node [below] {$\cong$} (m-2-3)
    (m-1-1) edge node [left] {$\Delta_{A,B}\otimes\Delta_{C,D}$} (m-2-1)
    (m-2-3) edge node [right] {$\mu_{A,C}\otimes\mu_{B,D}$} (m-1-3);
    \end{tikzpicture}
\end{equation}
must commute. We also require that the maps
\begin{align*}
    \mu_{S,\emptyset}&:\mathbf{p}[S]\otimes \mathbf{p}[\emptyset] \to \mathbf{p}[S], \\
    \mu_{\emptyset,S}&:\mathbf{p}[\emptyset]\otimes \mathbf{p}[S] \to \mathbf{p}[S], \\
    \Delta_{S,\emptyset}&:\mathbf{p}[S]\to\mathbf{p}[S]\otimes \mathbf{p}[\emptyset],\text{ and}  \\
    \Delta_{\emptyset,S}&:\mathbf{p}[S]\to\mathbf{p}[\emptyset]\otimes \mathbf{p}[S]
\end{align*}
are the canonical identifications.

\bigbreak

We call a Hopf monoid \emph{commutative} (respectively \emph{cocommutative}) if the left (respectively right) diagram commutes for all $S,T$.
\[
    \begin{tikzpicture}
  \matrix (m) [matrix of math nodes,row sep=3em,column sep=0em,minimum width=1em]
  {
     \mathbf{p}[S]\otimes\mathbf{p}[T] &{}& \mathbf{p}[T]\otimes \mathbf{p}[S] \\
     {} & \mathbf{p}[S \sqcup T] & {} \\};
  \path[-stealth]
    (m-1-1) edge node [above] {$\cong$} (m-1-3)
    (m-1-1) edge node [auto,swap] {$\mu_{S,T}$}(m-2-2)
    (m-1-3) edge node [auto] {$\mu_{T,S}$} (m-2-2);
    \end{tikzpicture} \hspace{.2in}
    \begin{tikzpicture}
  \matrix (m) [matrix of math nodes,row sep=3em,column sep=0em,minimum width=1em]
  {
     \mathbf{p}[S]\otimes\mathbf{p}[T] &{}& \mathbf{p}[T]\otimes \mathbf{p}[S] \\
     {} & \mathbf{p}[S \sqcup T] & {} \\};
  \path[-stealth]
    (m-1-1) edge node [above] {$\cong$} (m-1-3)
    (m-2-2) edge node [auto] {$\Delta_{S,T}$}(m-1-1)
    (m-2-2) edge node [auto,swap] {$\Delta_{T,S}$} (m-1-3);
    \end{tikzpicture}
\]
We will denote a Hopf monoid by the triple $(\mathbf{p},\mu,\Delta)$.

\begin{remark} The above definition of a Hopf monoid assumes that the species $\mathbf{p}$ is connected. Hopf monoids are defined for an arbitrary vector species in \cite[Section~8.3]{MR2724388}. For our purposes we only need to consider connected Hopf monoids, which gives us the unit, counit, and antipode for free (see \cite[Section~8.4]{MR2724388} for details). There is a more rigorous treatment of connected Hopf monoids in \cite[Section~2.2]{marbergselfdual}.
\end{remark}

\subsection{The Hopf monoid on posets}

For a finite set $S$, let $\mathbf{p}[S]$ be the $\mathbb{F}$-vector space with basis
\[
        \{x_\mathcal{P} \mid \mathcal{P} \text{ is a poset on }S\}
\]
and $\mathbf{p}[\emptyset]=\mathbb{F}$. Given a bijection $\sigma:S \to T$, define
\begin{align*}
        \mathbf{p}[\sigma]:\mathbf{p}[S] &\to \mathbf{p}[T] \\
        x_\mathcal{P} &\mapsto x_{\sigma\mathcal{P}},
\end{align*}
where $\sigma\mathcal{P}$ is the poset on $T$ with $t_1 \prec_{\sigma\mathcal{P}} t_2$ if and only if $\sigma^{-1}(t_1) \prec_{\mathcal{P}} \sigma^{-1}(t_2)$. This defines a vector species.

\bigbreak

Let $S$ and $T$ be disjoint finite sets. If $\mathcal{P}$ is a poset on $S$ and $\mathcal{Q}$ is a poset on $T$, define $\mathcal{P} \cdot \mathcal{Q}$ to be the poset on $S \sqcup T$ with $x \prec_{\mathcal{P} \cdot \mathcal{Q}} y$ if and only if
\begin{enumerate}
\item $x,y \in S$ and $x \prec_{\mathcal{P}} y$,
\item $x,y \in T$ and $x \prec_{\mathcal{Q}} y$, or
\item $x \in S$ and $y \in T$.
\end{enumerate}

\begin{example} Let $S = \{a,b,c,d\}$ and $T = \{x,y,z\}$ with posets $\mathcal{P}$ and $\mathcal{Q}$ given by the Hasse diagrams
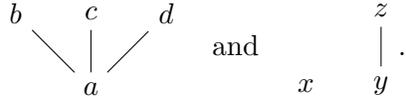
\[
    \begin{tikzpicture}[baseline={([yshift=-.5ex]current bounding box.center)}]
    \node (a) at (0,0) {$a$};
    \node (b) at (-1,1) {$b$};
    \node (c) at (0,1) {$c$};
    \node (d) at (1,1) {$d$};
    \draw (a) -- (b);
    \draw (a) -- (c);
    \draw (a) -- (d);
\end{tikzpicture} \quad \text{and} \quad
\begin{tikzpicture}[baseline={([yshift=-.5ex]current bounding box.center)}]
    \node (x) at (-.5,0) {$x$};
    \node (y) at (.5,0) {$y$};
    \node (z) at (.5,1) {$z$};
    \draw (y) -- (z);
\end{tikzpicture}.
\]
The Hasse diagram of $\mathcal{P}\cdot \mathcal{Q}$ is
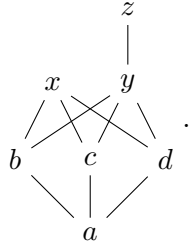
\[
    \begin{tikzpicture}[baseline={([yshift=-.5ex]current bounding box.center)}]
    \node (a) at (0,0) {$a$};
    \node (b) at (-1,1) {$b$};
    \node (c) at (0,1) {$c$};
    \node (d) at (1,1) {$d$};
    \node (x) at (-.5,2) {$x$};
    \node (y) at (.5,2) {$y$};
    \node (z) at (.5,3) {$z$};
    \draw (a) -- (b)--(x);
    \draw (a) -- (c)--(x);
    \draw (a) -- (d)--(x);
    \draw (b) -- (y)--(z);
    \draw (c) -- (y);
    \draw (d) -- (y);
\end{tikzpicture}.
\]
\end{example}

For any subset $S \subseteq I$, we can define the restriction of a poset $\mathcal{P}$ of $I$ to a poset $\mathcal{P}|_S$ of $S$ by $x \preceq_{\mathcal{P}|_S}y$ if and only if $x,y \in S$ and $x \preceq_{\mathcal{P}}y$.

\begin{example} Let $I = \{a,b,c,d\}$ and $\mathcal{P}$ be the poset with Hasse diagram
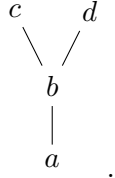
\[
\begin{tikzpicture}
    \node (a) at (0,0) {$a$};
    \node (b) at (0,1) {$b$};
    \node (c) at (-.5,2) {$c$};
    \node (d) at (.5,2) {$d$};
    \draw (a) -- (b);
    \draw (b) -- (c);
    \draw (b) -- (d);
\end{tikzpicture}.
\]
If $S = \{a,c,d\}$, then the Hasse diagram of $\mathcal{P}|_S$ is
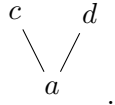
\[
\begin{tikzpicture}
    \node (a) at (0,0) {$a$};
    \node (c) at (-.5,1) {$c$};
    \node (d) at (.5,1) {$d$};
    \draw (a) -- (c);
    \draw (a) -- (d);
\end{tikzpicture}.
\]
\end{example}

For two disjoint finite sets $S$ and $T$, define
\[
        \mu_{S,T}(x_\mathcal{P}\otimes x_\mathcal{Q}) = x_{\mathcal{P}\cdot\mathcal{Q}}
\]
and
\[
        \Delta_{S,T}(x_\mathcal{P}) = x_{\mathcal{P}|_S} \otimes x_{\mathcal{P}|_T},
\]
then extend by linearity.

\begin{proposition}[{\cite[Section 13.1.1]{MR2724388}}] The product and coproduct defined above make $(\mathbf{p},\mu,\Delta)$ a Hopf monoid. This monoid is connected and cocommutative but not commutative.
\end{proposition}

\begin{remark} This Hopf monoid structure on $\mathbf{p}$ is the same as that of the dual Hopf monoid $\mathbf{P}^*$ constructed by Aguiar--Mahajan in \cite[Section 13.1.1]{MR2724388} with respect to a different basis.
\end{remark}

\subsection{The free Hopf monoid on a species}

Given a species $\mathbf{q}$ with $\mathbf{q}[\emptyset]=0$, we define the \emph{free Hopf monoid on} $\mathbf{q}$, denoted $\mathcal{T}(\mathbf{q})$, as follows. The underlying species structure has $\mathcal{T}(\mathbf{q})[\emptyset]=\mathbb{F}$, and for $I$ nonempty,
\[
        \mathcal{T}(\mathbf{q})[I] = \bigoplus \mathbf{q}[S_1]\otimes \mathbf{q}[S_2]\otimes \hdots \mathbf{q}[S_k],
\]
where the sum is over all ordered set partitions $(S_1,S_2,\hdots,S_k)$ of $I$ (in particular, each $S_i$ must be nonempty). The product on $\mathcal{T}(\mathbf{q})$ is given by concatenation; that is, if $x = x_1\otimes x_2 \otimes \hdots \otimes x_k \in \mathcal{T}(\mathbf{q})[S]$ and $y = y_1\otimes y_2 \otimes \hdots \otimes y_m \in \mathcal{T}(\mathbf{q})[T]$, then
\[
        \mu_{S,T}(x \otimes y) = x_1\otimes x_2 \otimes \hdots x_k \otimes y_1\otimes y_2 \otimes \hdots \otimes y_m.
\]
The coproduct on $\mathcal{T}(\mathbf{q})$ is defined by
\[
        \Delta_{S,T}(x) = 0
\]
for all $x \in \mathbf{q}[I]$ and all decompositions $I = S \sqcup T$ into nonempty subsets. Note that $\mathcal{T}(\mathbf{q})$ is a connected, cocommutative Hopf monoid. For more detail on the free Hopf monoid and other universal constructions of Hopf monoids, see \cite[Chapter 11]{MR2724388}.

\bigbreak

The following proposition allows us to show that a given Hopf monoid is free without consideration of the coproduct.

\begin{proposition}[{\cite[Proposition 23]{MR3117506}}]\label{propfree} Let $\mathbb{F}$ be a field of characteristic 0. Let $\mathbf{p}$ be a connected, cocommutative Hopf monoid over $\mathbb{F}$, and let $\mathbf{q}$ be a species over $\mathbb{F}$ such that $\mathbf{q}[\emptyset]=0$. If $\mathbf{p}$ and $\mathcal{T}(\mathbf{q})$ are isomorphic as monoids, then $\mathbf{p}$ and $\mathcal{T}(\mathbf{q})$ are isomorphic as Hopf monoids (although perhaps by a different isomorphism).
\end{proposition}

In Section~\ref{subsecfree} we will use this result to show that the Hopf monoid of nonnesting supercharacters is free.

\section{The nonnesting supercharacter theory of the unitriangular matrices}\label{secnnutn}

In this section we present a supercharacter theory of $UT_n(\mathbb{F}_q)$ that is coarser than the algebra group supercharacter theory and has supercharacters and superclasses indexed by the nonnesting $\mathbb{F}_q$-set partitions. This supercharacter theory is mentioned by the author in \cite[Section~7]{andrews2}, although the construction we present here is quite different.

\bigbreak

Let $\eta$ be an $\mathbb{F}_q$-set partition, and define
\begin{align*}
        \text{sml}(\eta) &= \{i \overset{a}{\frown}j \in \eta \mid \text{ there are no } k\overset{b}{\frown}l \in \eta \text{ with } i<k<l<j\} \text{ and} \\
        \text{big}(\eta) &= \{i \overset{a}{\frown}j \in \eta \mid \text{ there are no } k\overset{b}{\frown}l \in \eta \text{ with } k<i<j<l\}.
\end{align*}
In other words, if $\eta$ contains two arcs of the form
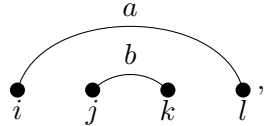
\[
        \begin{tikzpicture}[baseline={([yshift=-2ex]current bounding box.center)}]
	\fill (0,0) circle (.1) node[below]{$i$};
    \fill (1,0) circle (.1) node[below]{$j$};
    \fill (2,0) circle (.1) node[below]{$k$};
    \fill (3,0) circle (.1) node[below]{$l$};
    \draw (0,0) to [out=75, in=105] node[above]{$a$} (3,0);
    \draw (1,0) to [out=45, in=135] node[above]{$b$}(2,0);
    \end{tikzpicture},
\]
then the arc $i \overset{a}{\frown}l$ is removed in $\text{sml}(\eta)$ and the arc $j \overset{b}{\frown} k$ is removed in $\text{big}(\eta)$.

\begin{example}\label{exbigsml} If $n=12$ and
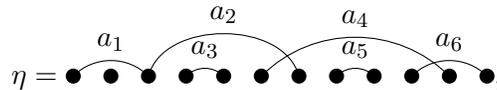
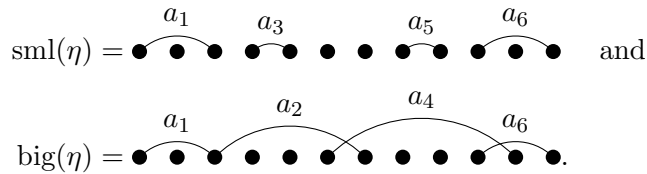
\[
        \eta = \begin{tikzpicture}
	\fill (0,0) circle (.1);
    \fill (.5,0) circle (.1);
    \fill (1,0) circle (.1);
    \fill (1.5,0) circle (.1);
    \fill (2,0) circle (.1);
    \fill (2.5,0) circle (.1);
    \fill (3,0) circle (.1);
    \fill (3.5,0) circle (.1);
    \fill (4,0) circle (.1);
    \fill (4.5,0) circle (.1);
    \fill (5,0) circle (.1);
    \fill (5.5,0) circle (.1);
    \draw (0,0) to [out=45, in=135] node[above] {$a_1$} (1,0);
	\draw (1,0) to [out=75, in=105] node[above] {$a_2$} (3,0);
    \draw (1.5,0) to [out=45, in=135] node[above] {$a_3$} (2,0);
	\draw (2.5,0) to [out=45, in=135] node[above] {$a_4$} (5,0);
    \draw (3.5,0) to [out=45, in=135] node[above] {$a_5$} (4,0);
	\draw (4.5,0) to [out=45, in=135] node[above] {$a_6$} (5.5,0);
    \end{tikzpicture},
\]
then we have
\begin{align*}
\text{sml}(\eta) &= \begin{tikzpicture}
	\fill (0,0) circle (.1);
    \fill (.5,0) circle (.1);
    \fill (1,0) circle (.1);
    \fill (1.5,0) circle (.1);
    \fill (2,0) circle (.1);
    \fill (2.5,0) circle (.1);
    \fill (3,0) circle (.1);
    \fill (3.5,0) circle (.1);
    \fill (4,0) circle (.1);
    \fill (4.5,0) circle (.1);
    \fill (5,0) circle (.1);
    \fill (5.5,0) circle (.1);
    \draw (0,0) to [out=45, in=135] node[above] {$a_1$} (1,0);
    \draw (1.5,0) to [out=45, in=135] node[above] {$a_3$} (2,0);
    \draw (3.5,0) to [out=45, in=135] node[above] {$a_5$} (4,0);
	\draw (4.5,0) to [out=45, in=135] node[above] {$a_6$} (5.5,0);
    \end{tikzpicture} \quad\text{ and} \\
        \text{big}(\eta) &= \begin{tikzpicture}
	\fill (0,0) circle (.1);
    \fill (.5,0) circle (.1);
    \fill (1,0) circle (.1);
    \fill (1.5,0) circle (.1);
    \fill (2,0) circle (.1);
    \fill (2.5,0) circle (.1);
    \fill (3,0) circle (.1);
    \fill (3.5,0) circle (.1);
    \fill (4,0) circle (.1);
    \fill (4.5,0) circle (.1);
    \fill (5,0) circle (.1);
    \fill (5.5,0) circle (.1);
    \draw (0,0) to [out=45, in=135] node[above] {$a_1$} (1,0);
	\draw (1,0) to [out=45, in=135] node[above] {$a_2$} (3,0);
	\draw (2.5,0) to [out=45, in=135] node[above] {$a_4$} (5,0);
	\draw (4.5,0) to [out=45, in=135] node[above] {$a_6$} (5.5,0);
    \end{tikzpicture}.
\end{align*}
\end{example}
Note that both $\text{sml}(\eta)$ and $\text{big}(\eta)$ are nonnesting $\mathbb{F}_q$-set partitions. These methods of producing a nonnesting $\mathbb{F}_q$-set partition from an arbitrary $\mathbb{F}_q$-set partition both define equivalence relations on the set of $\mathbb{F}_q$-set partitions, and in both cases the equivalence classes are indexed by the nonnesting $\mathbb{F}_q$-set partitions.  For a nonnesting $\mathbb{F}_q$-set partition $\eta$, let
\begin{align*}
        K_{[\eta]} &= \bigcup_{\text{sml}(\nu) = \eta} K_{\nu} \text{ and} \\
        \chi_{[\eta]} &= \sum_{\text{big}(\nu)=\eta} \chi_{\nu},
\end{align*}
where $K_\nu$ and $\chi_\nu$ are as in Corollary~\ref{cortypeaschscl}. Note that we have
\[
        K_{[\eta]} = \left\{ g \in UT_n(\mathbb{F}_q) \;\left|\;\begin{array}{l} g_{ij} = a \text{ for all } i\overset{a}{\frown} j \in \eta, \text{ and } g_{kl} = 0 \text{ unless}\\
        \text{there exists } i\overset{a}{\frown} j \in \eta \text{ with } k \leq i < j \leq l\end{array}\right.\right\}.
\]

\begin{remark} The reason that we use $\text{sml}(\eta)$ to define the superclasses and $\text{big}(\eta)$ to define the supercharacters has to do with how $UT_n(\mathbb{F}_q)$ acts on $\frak{ut}_n(\mathbb{F}_q)$ and $\frak{ut}_n(\mathbb{F}_q)^*$ by left and right multiplication. Loosely speaking, the action on $\frak{ut}_n(\mathbb{F}_q)$ is by row and column reduction, where a row can be added to any row above it and a column can be added to any column to its right. This corresponds to removing larger arcs from a set partition. The action on $\frak{ut}_n(\mathbb{F}_q)^*$ is also by row and column reduction, but in this case a row can be added to any row below it and a column can be added to any column to its left. This corresponds to removing smaller arcs from a set partition.
\end{remark}

In order to provide an alternative description of the characters $\chi_{[\eta]}$, we define a subgroup
\begin{equation}\label{eqtueta}
        U_\eta = \bigg\{g \in UT_n(\mathbb{F}_q) \;\bigg|\;\begin{array}{l} g_{ij} = 0 \text{ if there exists }k \overset{a}{\frown}l \in \eta \text{ with } \\(i,j) \neq (k,l) \text{ and } k \leq i<j \leq l\end{array}\bigg\}.
\end{equation}

\begin{example} Let
\[
        \eta = \begin{tikzpicture}
	\fill (0,0) circle (.1);
    \fill (.5,0) circle (.1);
    \fill (1,0) circle (.1);
    \fill (1.5,0) circle (.1);
    \fill (2,0) circle (.1);
    \fill (2.5,0) circle (.1);
    \fill (3,0) circle (.1);
    \fill (3.5,0) circle (.1);
    \fill (4,0) circle (.1);
    \fill (4.5,0) circle (.1);
    \fill (5,0) circle (.1);
    \fill (5.5,0) circle (.1);
    \draw (0,0) to [out=45, in=135] node[above] {$a_1$} (1,0);
	\draw (1,0) to [out=75, in=105] node[above] {$a_2$} (3,0);
	\draw (2.5,0) to [out=45, in=135] node[above] {$a_3$} (5,0);
	\draw (4.5,0) to [out=45, in=135] node[above] {$a_4$} (5.5,0);
    \end{tikzpicture};
\]
then we have
\[
        U_\eta = \left\{\left(\begin{array}{llllllllllll}
        1 & 0 & \bm{*} & * & * & * & * & * & * & * & * & * \\
        0 & 1 & 0 & * & * & * & * & * & * & * & * & * \\
        0 & 0 & 1 & 0 & 0 & 0 & \bm{*} & * & * & * & * & * \\
        0 & 0 & 0 & 1 & 0 & 0 & 0 & * & * & * & * & * \\
        0 & 0 & 0 & 0 & 1 & 0 & 0 & * & * & * & * & * \\
        0 & 0 & 0 & 0 & 0 & 1 & 0 & 0 & 0 & 0 & \bm{*} & * \\
        0 & 0 & 0 & 0 & 0 & 0 & 1 & 0 & 0 & 0 & 0 & * \\
        0 & 0 & 0 & 0 & 0 & 0 & 0 & 1 & 0 & 0 & 0 & * \\
        0 & 0 & 0 & 0 & 0 & 0 & 0 & 0 & 1 & 0 & 0 & * \\
        0 & 0 & 0 & 0 & 0 & 0 & 0 & 0 & 0 & 1 & 0 & \bm{*} \\
        0 & 0 & 0 & 0 & 0 & 0 & 0 & 0 & 0 & 0 & 1 & 0 \\
        0 & 0 & 0 & 0 & 0 & 0 & 0 & 0 & 0 & 0 & 0 & 1
        \end{array}\right)\right\},
\]
where the bold entries correspond to the arcs of $\eta$.
\end{example}

\begin{lemma}\label{lemchietaind} Let $U_\eta$ be as in Equation~\ref{eqtueta}; then
\[
        \chi_{[\eta]} = \textup{Ind}_{U_\eta}^{UT_n(\mathbb{F}_q)} \textup{Res}_{U_\eta}^{UT_n(\mathbb{F}_q)}(\theta \circ \lambda_\eta \circ f),
\]
where $\lambda_\eta(x) = \sum_{i \overset{a}{\frown}j \in \eta} ax_{ij}$, as in Proposition~\ref{lemtypeascl}.
\end{lemma}

\begin{proof} From the definition of $U_\eta$, we have that $U_\eta$ is a normal pattern subgroup of $UT_n(\mathbb{F}_q)$, and furthermore that $\textup{Res}_{U_\eta}^{UT_n(\mathbb{F}_q)}(\theta \circ \lambda_\eta \circ f)$ is invariant under the conjugation action of $UT_n(\mathbb{F}_q)$. By \cite[Lemma~4.5]{andrews1}, we have that
\begin{align*}
        \textup{Ind}_{U_\eta}^{UT_n(\mathbb{F}_q)} \textup{Res}_{U_\eta}^{UT_n(\mathbb{F}_q)}(\theta \circ \lambda_\eta \circ f)(g) & = \left\{\begin{array}{ll}
        \frac{|UT_n(\mathbb{F}_q)|}{|U_\eta|}(\theta \circ \lambda_\eta \circ f)(g) &\quad \text{if } g \in U_\eta, \\
        0 & \quad \text{otherwise,}\end{array}\right. \\
        & = \sum_{\substack{\mu \in \frak{ut}_n(\mathbb{F}_q)^* \\ \mu|_{\frak{u}_\eta} = \lambda_\eta|_{\frak{u}_\eta}}}\theta \circ \mu \circ f,
\end{align*}
where $\mu|_{\frak{u}_\eta}$ denotes the restriction of $\mu$ to $\frak{u}_\eta$. At the same time, we have
\[
         \chi_{[\eta]} = \sum_{\substack{\mu \in \frak{ut}_n(\mathbb{F}_q)^* \\ \mu|_{\frak{u}_\eta} = \lambda_\eta|_{\frak{u}_\eta}}}\theta \circ \mu \circ f.
\]
\end{proof}

As a corollary, we can calculate the dimensions of the characters $\chi_{[\eta]}$.

\begin{corollary}\label{cordimnnsch} Let $\eta$ be a nonnesting $\mathbb{F}_q$-set partition. Then $\chi_{[\eta]}(1) = |U:U_\eta|$, and we have
\begin{align*}
        |U:U_\eta| &= q^{\left|\left\{(i,j) \;\bigg|\; \begin{array}{l}i<j\text{ and there exists }k \overset{a}{\frown}l \in \eta \text{ with } \\(i,j) \neq (k,l) \text{ and } k \leq i<j \leq l\end{array}\right\}\right|}.
\end{align*}
\end{corollary}

We now calculate the values of the characters $\chi_{[\eta]}$ on the elements of the sets $K_{[\nu]}$.

\begin{proposition}\label{nnschtable} Let $\eta$ and $\nu$ be nonnesting $\mathbb{F}_q$-set partitions and let $g \in K_{[\nu]}$. Then
\[
        \chi_{[\eta]}(g) = \left\{\begin{array}{ll}
        \chi_{[\eta]}(1)\displaystyle\prod_{\substack{i\overset{a}{\frown}j \in \eta \\i \overset{b}{\frown}j \in \nu}}\theta(ab) & \quad\begin{array}{l}\text{if there are no } i \overset{a}{\frown} j \in \eta \text{ and } k \overset{b}{\frown} l \in \nu \\ \text{with } (i,j) \neq (k,l) \text{and } i\leq k < l \leq j,\end{array} \\
        0 &\quad\text{otherwise.}\end{array}\right.
\]
\end{proposition}

\begin{proof} There are no $i \overset{a}{\frown} j \in \eta$ and $k \overset{b}{\frown} l \in \nu$ with $(i,j) \neq (k,l)$ and $i\leq k < l \leq j$ exactly when $K_{[\nu]} \subseteq U_\eta$. Assume that $g \in U_\eta$; then if $i \overset{a}{\frown} j \in \eta$ and $g_{ij} \neq 0$, we have that $i \overset{g_{ij}}{\frown} j \in \nu$. The result follows from Lemma~\ref{lemchietaind}.
\end{proof}

We can now show that we have constructed a supercharacter theory of $UT_n(\mathbb{F}_q)$.

\begin{theorem}\label{theoremnnsct} The sets
\begin{align*}
        &\{\chi_{[\eta]} \mid \eta \text{ is a nonnesting }\mathbb{F}_q\text{-set partition}\} \quad \text{and} \\ &\{K_{[\nu]} \mid \nu \text{ is a nonnesting }\mathbb{F}_q\text{-set partition}\}
\end{align*}
are the supercharacters and superclasses for a supercharacter theory of $UT_n(\mathbb{F}_q)$.
\end{theorem}
\begin{proof} As both sets are indexed by nonnesting $\mathbb{F}_q$-set partitions, we have (SCT1). Proposition~\ref{nnschtable} demonstrates that (SCT2) holds, and (SCT3) follows from the fact that each algebra group supercharacter is a constituent of exactly one $\chi_{[\eta]}$.
\end{proof}

We call this the \emph{nonnesting supercharacter theory} of $UT_n(\mathbb{F}_q)$.

\section{Nonnesting supercharacter theories of pattern groups}\label{secnnp}

Let $I$ be a finite set, and let $\mathcal{P}$ be a poset on $I$. Recall that the pattern group $U_\mathcal{P}$ has a supercharacter theory with superclasses
\[
        K_g = \{h \in U_\mathcal{P} \mid f(h) \in U_\mathcal{P}f(g)U_\mathcal{P}\}
\]
and supercharacters
\[
        \chi_\lambda = \sum_{\mu \in U_\mathcal{P}\lambda U_\mathcal{P}} \theta \circ \mu \circ f.
\]
Unfortunately, there is no known indexing set for the supercharacters and superclasses of an arbitrary pattern group. In \cite[Corollary 4.5]{MR2855781}, Halasi--P\'alfy prove that there exists a poset $\mathcal{P}$ such that $U_\mathcal{P}$ has nilpotency class 2 and the number of conjugacy classes of $U_\mathcal{P}$ cannot be described by finitely many polynomials in $q$. For pattern groups of nilpotency class 2, however, the algebra group supercharacters are just the irreducible characters (this is a consequence of \cite[Corollary 5.1]{MR2491890}). It follows that the number of algebra group supercharacters of $U_\mathcal{P}$ is not in general polynomial in $q$.

\bigbreak

In this section, we present a supercharacter theory of $U_\mathcal{P}$ that is coarser than the algebra group supercharacter theory and is analogous to the nonnesting supercharacter theory of $UT_n(\mathbb{F}_q)$. This supercharacter theory has supercharacters and superclasses indexed by the nonnesting $(\mathbb{F}_q,\mathcal{P})$-set partitions, and as such the number of supercharacters for a fixed poset is polynomial in $q$. The results and proofs in this section are almost identical to those of Section~\ref{secnnutn}; because we do not have a nice indexing set for the algebra group supercharacters and superclasses of $U_\mathcal{P}$, however, we must use a different approach to construct the supercharacters and superclasses.

\bigbreak

For an element $g \in U_\mathcal{P}$, let
\[
        \text{sml}(g) = \bigg\{i\overset{g_{ij}}{\frown}j \;\bigg|\; \begin{array}{l} g_{ij} \neq 0 \text{ and if } (k,l) \neq (i,j) \\ \text{ and } i \preceq_\mathcal{P} k\prec_\mathcal{P}l \preceq_\mathcal{P} j,\text{ then } g_{kl}=0\end{array}\bigg\}.
\]
Similarly, if $\lambda \in \frak{u}_\mathcal{P}^*$, define
\[
        \text{big}(\lambda) = \bigg\{i\overset{\lambda(e_{ij})}{\frown}j \;\bigg|\; \begin{array}{l} \lambda(e_{ij}) \neq 0 \text{ and if } (k,l) \neq (i,j) \\ \text{ and } k \preceq_\mathcal{P} i\prec_\mathcal{P}j \preceq_\mathcal{P} l,\text{ then } \lambda(e_{kl})=0\end{array}\bigg\},
\]
where $e_{ij}\in \frak{gl}_n(\mathbb{F}_q)$ is the matrix with a 1 in the $(i,j)$ entry and zeroes elsewhere.

\bigbreak

\begin{example} Let $\mathcal{P}$ be the poset with Hasse diagram
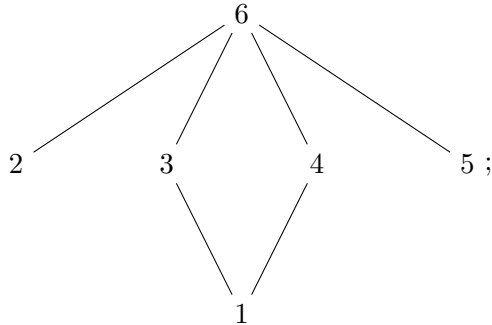
\[
    \begin{tikzpicture}[baseline={([yshift=-.5ex]current bounding box.center)}]
    \node (e) at (0,2) {$6$};
    \node (a) at (-3,0) {$2$};
    \node (b) at (-1,0) {$3$};
    \node (c) at (1,0) {$4$};
    \node (d) at (3,0) {$5$};
    \node (f) at (0,-2) {$1$};
    \draw (e) -- (a);
    \draw (e) -- (b) -- (f) -- (c) -- (e) -- (d);
\end{tikzpicture};
\]
then

\[
        U_\mathcal{P} = \left\{\left(\begin{array}{llllllllllll}
        1 & 0 & * & * & 0 & * \\
        0 & 1 & 0 & 0 & 0 & * \\
        0 & 0 & 1 & 0 & 0 & * \\
        0 & 0 & 0 & 1 & 0 & * \\
        0 & 0 & 0 & 0 & 1 & * \\
        0 & 0 & 0 & 0 & 0 & 1
        \end{array}\right)\right\}.
\]
For $a,b,c,e \in \mathbb{F}_q$, let
\[
        g = \left(\begin{array}{llllllllllll}
        1 & 0 & a & 0 & 0 & b \\
        0 & 1 & 0 & 0 & 0 & c \\
        0 & 0 & 1 & 0 & 0 & 0 \\
        0 & 0 & 0 & 1 & 0 & e \\
        0 & 0 & 0 & 0 & 1 & 0 \\
        0 & 0 & 0 & 0 & 0 & 1
        \end{array}\right); \quad
\text{then} \quad
    \text{sml}(g) = \begin{tikzpicture}
	\fill (0,0) circle (.1);
    \fill (.5,0) circle (.1);
    \fill (1,0) circle (.1);
    \fill (1.5,0) circle (.1);
    \fill (2,0) circle (.1);
    \fill (2.5,0) circle (.1);
    \draw (0,0) to [out=75, in=105] node[above] {$a$} (1,0);
	\draw (.5,0) to [out=75, in=105] node[above] {$c$} (2.5,0);
    \draw (1.5,0) to [out=25, in=155] node[above] {$e$} (2.5,0);
\end{tikzpicture}.
\]
If we define $\lambda \in \frak{u}_\mathcal{P}^*$ by
\[
        \lambda(x) = ax_{13}+bx_{16}+cx_{26}+ex_{46}, \quad
\text{then} \quad
    \text{big}(\lambda) = \begin{tikzpicture}
	\fill (0,0) circle (.1);
    \fill (.5,0) circle (.1);
    \fill (1,0) circle (.1);
    \fill (1.5,0) circle (.1);
    \fill (2,0) circle (.1);
    \fill (2.5,0) circle (.1);
    \draw (0,0) to [out=75, in=105] node[above] {$b$} (2.5,0);
	\draw (.5,0) to [out=25, in=155] node[above] {$c$} (2.5,0);
\end{tikzpicture}.
\]
\end{example}

Note that $\text{sml}(g)$ and $\text{big}(\lambda)$ are nonnesting $(\mathbb{F}_q,\mathcal{P})$-set partitions for all $g \in U_\mathcal{P}$ and $\lambda \in \frak{u}_\mathcal{P}^*$.

\begin{remark} If $\mathcal{P}$ is the usual linear order on $[n]$ and $g \in U_\mathcal{P}$ is in the algebra group superclass $K_\eta$, then $\text{sml}(g) = \text{sml}(\eta)$ (a similar statement can be made for functionals).
\end{remark}

For nonnesting $(\mathbb{F}_q,\mathcal{P})$-set partitions $\eta$ and $\nu$, define
\[
        K_\nu = \{g \in U_\mathcal{P} \mid \text{sml}(g) = \nu\}
\]
and
\[
        \chi_\eta = \sum_{\substack{\lambda \in \frak{u}_\mathcal{P}^* \\ \text{big}(\lambda) = \eta}} \theta \circ \lambda \circ f.
\]
Analogous to the case of $UT_n(\mathbb{F}_q)$, we have that
\[
        K_{\eta} = \left\{ g \in U_\mathcal{P} \;\left|\;\begin{array}{l} g_{ij} = a \text{ for all } i\overset{a}{\frown} j \in \eta, \text{ and } g_{kl} = 0 \text{ unless}\\
        \text{there exists } i\overset{a}{\frown} j \in \eta \text{ with } k \preceq i \prec j \preceq l\end{array}\right.\right\}.
    \]

\begin{lemma}\label{anotherlemma} We have that
\begin{enumerate}
\item if $g \in U_\mathcal{P}$ and $f(h) \in U_\mathcal{P}f(g)U_\mathcal{P}$, then $\text{sml}(g)=\text{sml}(h)$; and
\item if $\lambda \in \frak{u}_\mathcal{P}^*$ and $\mu \in U_\mathcal{P}\lambda U_\mathcal{P}$, then $\text{big}(\lambda) = \text{big}(\mu)$.
\end{enumerate}
\end{lemma}

\begin{proof} We prove (1), and mention that the proof of (2) is similar. Let $x \in \frak{u}_\mathcal{P}$ and $u,v \in U_\mathcal{P}$; then
\[
        (uxv)_{ij} = \sum_{\substack{k,l \text{ such that} \\ i \preceq_\mathcal{P} k\prec_\mathcal{P}l \preceq_\mathcal{P} j}} u_{ik}x_{kl}v_{lj}.
\]
Suppose that $g_{ij} \neq 0$, and that for all $(k,l) \neq (i,j)$ with $i \preceq_\mathcal{P} k\prec_\mathcal{P}l \preceq_\mathcal{P} j$, we have $g_{kl}=0$. If $f(h) \in U_\mathcal{P} f(g) U_\mathcal{P}$, then by the above calculation
\begin{enumerate}
\item $h_{ij} = g_{ij}$ and
\item for all $(k,l) \neq (i,j)$ with $i \preceq_\mathcal{P} k\prec_\mathcal{P}l \preceq_\mathcal{P} j$, we have $h_{kl}=0$.
\end{enumerate}
It follows that $\text{sml}(g) \subseteq \text{sml}(h)$, and the same argument shows that $\text{sml}(h) \subseteq \text{sml}(g)$.
\end{proof}

There is an important corollary of Lemma~\ref{anotherlemma}.

\begin{corollary}\label{importantcorollary}
The sets $K_\nu$ are unions of algebra group superclasses, and the functions $\chi_\eta$ are sums of algebra group supercharacters. In particular, the functions $\chi_\eta$ are characters of $U_\mathcal{P}$.
\end{corollary}

In order to provide an alternative description of the characters $\chi_\eta$, we define a subgroup
\begin{equation}
        U_\eta = \bigg\{g \in U_\mathcal{P} \;\bigg|\;\begin{array}{l} g_{ij} = 0 \text{ if there exists }k \overset{a}{\frown}l \in \eta \text{ with } \\(i,j) \neq (k,l) \text{ and } k \preceq_\mathcal{P} i\prec_\mathcal{P}j\preceq_\mathcal{P} l\end{array}\bigg\}.
\end{equation}

\begin{lemma}\label{ueta} Let $\eta$ be a nonnesting $(\mathbb{F}_q,\mathcal{P})$-set partition and let $\chi_\eta$ be as above. Then
\[
        \chi_{\eta} = \textup{Ind}_{U_\eta}^{U_\mathcal{P}} \textup{Res}_{U_\eta}^{U_\mathcal{P}}(\theta \circ \lambda \circ f),
\]
where $\lambda \in \frak{u}_\mathcal{P}^*$ is any functional with $\text{big}(\lambda) = \eta$.
\end{lemma}

\begin{proof} From the definition of $U_\eta$, we have that $U_\eta$ is a normal pattern subgroup of $U_\mathcal{P}$, and furthermore that $\theta \circ \lambda \circ f$ is invariant under the conjugation action of $U_\mathcal{P}$. It follows that
\begin{align*}
        \textup{Ind}_{U_\eta}^{U_\mathcal{P}} \textup{Res}_{U_\eta}^{U_\mathcal{P}}(\theta \circ \lambda \circ f)(g) & = \left\{\begin{array}{ll}
        \frac{|U_\mathcal{P}|}{|U_\eta|}(\theta \circ \lambda \circ f)(g) &\quad \text{if } g \in U_\eta, \\
        0 & \quad \text{otherwise,}\end{array}\right. \\
        & = \sum_{\substack{\mu \in \frak{u}_\mathcal{P}^* \\ \mu|_{\frak{u}_\eta} = \lambda|_{\frak{u}_\eta}}}\theta \circ \mu \circ f
\end{align*}
by \cite[Lemma~4.5]{andrews1}. At the same time,
\[
         \chi_{\eta} = \sum_{\substack{\mu \in \frak{u}_\mathcal{P}^* \\ \text{big}(\mu) = \eta}}\theta \circ \mu \circ f,
\]
and
\begin{align*}
        \{\mu \in \frak{u}_\mathcal{P}^* \mid \text{big}(\mu) = \eta\}
        = \{\mu \in \frak{u}_\mathcal{P}^* \mid \mu|_{\frak{u}_\eta} = \lambda|_{\frak{u}_\eta}\}.
\end{align*}
\end{proof}

Lemma~\ref{ueta} allows us to calculate the dimensions of the characters $\chi_\eta$.

\begin{corollary}\label{nnpdim} Let $\eta$ be a nonnesting $(\mathbb{F}_q,\mathcal{P})$-set partition. Then $\chi_\eta(1) = |U^\mathcal{P}:U_\eta|$, and we have
\begin{align*}
        |U^{\mathcal{P}}:U_\eta| &= q^{\left|\left\{(i,j) \;\bigg|\; \begin{array}{l} i\prec_\mathcal{P}j \text{ and there exists }
        k \overset{a}{\frown} l \in \eta \text{ with } \\(i,j) \neq (k,l) \text{ and } k \preceq_\mathcal{P} i\prec_\mathcal{P}j\preceq_\mathcal{P} l \end{array}\right\}\right|}.
\end{align*}
\end{corollary}

We now calculate the values of the characters $\chi_\eta$ on the elements of the sets $K_\nu$.

\begin{proposition}\label{anotherprop} Let $\eta$ and $\nu$ be nonnesting $(\mathbb{F}_q,\mathcal{P})$-set partitions. If $g \in K_\nu$, then
\[
        \chi_\eta(g) = \left\{\begin{array}{ll}
        \chi_\eta(1)\displaystyle\prod_{\substack{i\overset{a}{\frown}j \in \eta \\
        i \overset{b}{\frown}{j} \in \nu}} \theta(ab) & \quad \begin{array}{l} \text{if there are no } i \overset{a}{\frown}j \in \eta \text{ and } k\overset{b}{\frown}l \in \nu \\ \text{with } (i,j) \neq (k,l) \text{ and }
        i \preceq_\mathcal{P} k\prec_\mathcal{P}l\preceq_\mathcal{P} j,\end{array} \\
        0 & \quad \text{otherwise.}\end{array}\right.
\]
\end{proposition}

\begin{proof} There will be no $i \overset{a}{\frown}j \in \eta$ and $k\overset{b}{\frown}l \in \nu$ with $(i,j) \neq (k,l)$ and $i \preceq_\mathcal{P} k\prec_\mathcal{P}l\preceq_\mathcal{P} j$ exactly when $K_\nu \subseteq U_\eta$. Note that if $g \in U_\eta$, $i \overset{a}{\frown}j \in \eta$, and $g_{ij} \neq 0 $, then $i \overset{g_{ij}}{\frown} j \in \text{sml}(g)$; as $g \in K_\nu$, we have $i \overset{g_{ij}}{\frown} j \in \nu$. Furthermore, if $\text{big}(\lambda) = \eta$ and $g \in U_\eta$, we have
\[
        (\theta \circ \lambda\circ f)(g) = \prod_{i \overset{a}{\frown}j \in \eta} \theta(ag_{ij}).
\]
The result follows from Lemma~\ref{ueta}.
\end{proof}

We can now prove that we have constructed a supercharacter theory of $U_\mathcal{P}$.

\begin{theorem}\label{nntheorem2}\label{theoremnnposetsct} The sets
\begin{align*}
        &\{K_\nu\mid \nu \text{ is a nonnesting } (\mathbb{F}_q,\mathcal{P})\text{-set partition}\} \text{ and} \\
        &\{\chi_\eta\mid \eta \text{ is a nonnesting } (\mathbb{F}_q,\mathcal{P})\text{-set partition}\}
\end{align*}
\noindent
are the superclasses and supercharacters for a supercharacter theory of $U_\mathcal{P}$.
\end{theorem}

\begin{proof} As both sets are indexed by the nonnesting $(\mathbb{F}_q,\mathcal{P})$-set partitions, we have (SCT1). Proposition~\ref{anotherprop} gives us (SCT2), and (SCT3) follows from Corollary~\ref{importantcorollary}.
\end{proof}

\begin{remark} If $\mathcal{P}$ is the usual linear order on $[n]$, then $U_\mathcal{P} = UT_n(\mathbb{F}_q)$. In this case, we recover Theorem~\ref{theoremnnsct} from Theorem~\ref{nntheorem2}.
\end{remark}

\section{The Hopf monoid on complex-valued functions of pattern groups}\label{sechmpg}

In \cite{MR3117506}, Aguiar--Bergeron--Thiem construct a Hopf monoid from functions on the groups of unitriangular matrices with entries in $\mathbb{F}_q$. In this section we generalize their construction to pattern groups and show that the nonnesting supercharacter theory defines a combinatorial Hopf monoid.

\subsection{Construction}

Our construction in this section mirrors that of \cite{MR3117506} for unitriangular matrices.

\bigbreak

There is a contravariant functor
\[
        \mathbf{f}:\mathbf{Set} \to \mathbf{Vec}
\]
from the category of sets to the category of complex vector spaces. This functor maps a set $X$ to the space of functions from $X$ to $\mathbb{C}$. There is a canonical isomorphism
\begin{align*}
        \varphi: \mathbf{f}(X \times Y) &\to \mathbf{f}(X) \otimes \mathbf{f}(Y) \\
        (\alpha,\beta) & \mapsto \alpha\otimes\beta.
\end{align*}

If $U_\mathcal{P}$ is a pattern group, then $\mathbf{f}(U_\mathcal{P})$ is the space of functions from $U_\mathcal{P}$ to $\mathbb{C}$. We use this to define a vector species as follows. For a finite set $I$, define
\[
        \mathbf{fp}[I] = \bigoplus_{\substack{ \mathcal{P} \text{ is a}\\ \text{poset on } I}} \mathbf{f}(U_\mathcal{P}).
\]
Any bijection $\sigma: I \to J$ induces an isomorphism $ U_{\sigma\mathcal{P}}\cong U_\mathcal{P}$, which by functoriality gives an isomorphism $\mathbf{f}(U_{\mathcal{P}}) \cong \mathbf{f}(U_{\sigma\mathcal{P}})$. This means that $\mathbf{fp}$ is in fact a vector species.

\bigbreak

Suppose that $S$ and $T$ are disjoint finite sets and that $\mathcal{P}$ and $\mathcal{Q}$ are posets on $S$ and $T$, respectively. There is a homomorphism
\begin{align*}
        \pi_{\mathcal{P},\mathcal{Q}}: U_{\mathcal{P}\cdot\mathcal{Q}} &\to U_{\mathcal{P}} \times U_{\mathcal{Q}} \\
        g & \mapsto (g_S,g_T),
\end{align*}
where $(g_S)_{ij} = g_{ij}$ for all $i,j \in S$.

\bigbreak

We use this homomorphism to define a product on $\mathbf{fp}$. Given posets $\mathcal{P}$ and $\mathcal{Q}$ on disjoint finite sets $S$ and $T$, define
\begin{align*}
        \mu_{\mathcal{P},\mathcal{Q}}:\mathbf{f}(U_\mathcal{P})\otimes\mathbf{f}(U_\mathcal{Q}) &\to \mathbf{f}(U_{\mathcal{P}\cdot\mathcal{Q}}) \\
        \alpha \otimes \beta &\mapsto (\alpha,\beta) \circ \pi_{\mathcal{P},\mathcal{Q}}.
\end{align*}
In other words, $\mu_{\mathcal{P},\mathcal{Q}} = \mathbf{f}(\pi_{\mathcal{P},\mathcal{Q}})\circ \varphi^{-1}$. Extending by linearity over all posets on $S$ and $T$, we get a map
\[
        \mu_{S,T}:\mathbf{fp}[S]\otimes\mathbf{fp}[T] \to \mathbf{fp}[S \sqcup T].
\]
\begin{remark} The map $\mu_{\mathcal{P},\mathcal{Q}}$ is the inflation map with respect to the projection $\pi$. That is,
\[
        \mu_{\mathcal{P},\mathcal{Q}} = \text{Inf}_{U_\mathcal{P} \times U_\mathcal{Q}}^{U_{\mathcal{P}\cdot\mathcal{Q}}}\circ\varphi^{-1}.
\]
\end{remark}

We can also use a homomorphism to define the coproduct. Suppose that $S$ and $T$ are disjoint and that $\mathcal{P}$ is a poset on $S \sqcup T$. Define a homomorphism
\[
        \sigma_{\mathcal{P}|_S,\mathcal{P}|_T}:U_{\mathcal{P}|_S} \times U_{\mathcal{P}|_T} \to U_{\mathcal{P}}
\]
by
\[
        (\sigma_{\mathcal{P}|_S,\mathcal{P}|_T}(g,h))_{i,j} = \left\{\begin{array}{ll}
        g_{ij} & \quad \text{if } i,j \in S, \\
        h_{ij} & \quad \text{if }i,j \in T, \\
        0 & \quad \text{otherwise.} \end{array}\right.
\]
Let
\begin{align*}
        \Delta_{\mathcal{P}|_S,\mathcal{P}|_T}: \mathbf{f}(U_{\mathcal{P}}) &\to \mathbf{f}(U_{\mathcal{P}_S})\otimes \mathbf{f}(U_{\mathcal{P}_T}) \\
        \alpha & \mapsto \varphi(\alpha \circ \sigma_{\mathcal{P}|_S,\mathcal{P}|_T});
\end{align*}
in other words, $\Delta_{\mathcal{P}|_S,\mathcal{P}|_T} = \varphi \circ \mathbf{f}(\sigma_{\mathcal{P}|_S,\mathcal{P}|_T})$. Extending by linearity over all posets on $S\sqcup T$, we obtain a map
\[
        \Delta_{S,T}: \mathbf{fp}[S \sqcup T] \to \mathbf{fp}[S]\otimes \mathbf{fp}[T].
\]
\begin{remark} The map $\Delta_{\mathcal{P}|_S,\mathcal{P}|_T}$ is the restriction map; that is,
\[
        \Delta_{\mathcal{P}|_S,\mathcal{P}|_T} = \varphi\circ \text{Res}_{U_{\mathcal{P}_S}\times U_{\mathcal{P}_T}}^{U_\mathcal{P}}.
\]
\end{remark}

\begin{theorem} Let $\mu$ and $\Delta$ be the collections of all $\mu_{S,T}$ and $\Delta_{S,T}$; then the triple $(\mathbf{fp},\mu,\Delta)$ is a connected, cocommutative Hopf monoid.
\end{theorem}

\begin{proof} We will only prove Condition~\ref{hopfcompatible} as the proofs of the other conditions are analogous. Let $(S_1,S_2)$ and $(T_1,T_2)$ be pairs of disjoint finite sets such that $S_1\sqcup S_2 = T_1 \sqcup T_2 =I$. Let
\[
        A = S_1 \cap T_1, \quad B = S_1\cap T_2, \quad C = S_2 \cap T_1, \text{ and } D = S_2 \cap T_2.
\]
Suppose that $\mathcal{P}_1$ and $\mathcal{P}_2$ are posets on $S_1$ and $S_2$, and $\mathcal{P} = \mathcal{P}_1\cdot\mathcal{P}_2$. We need to show that the diagram
\begin{equation*}
\begin{tikzpicture}[font=\tiny]
  \matrix (m) [matrix of math nodes,row sep=3em,column sep=0em,minimum width=2em]
  {
     \mathbf{f}[U_{\mathcal{P}_1}]\otimes\mathbf{f}[U_{\mathcal{P}_2}] & \mathbf{f}[U_{\mathcal{P}}] & \mathbf{f}[U_{\mathcal{P}|_{T_1}}] \otimes \mathbf{f}[U_{\mathcal{P}|_{T_2}}] \\
     \mathbf{f}[U_{\mathcal{P}_1|_A}]\otimes\mathbf{f}[U_{\mathcal{P}_1|_B}]\otimes \mathbf{f}[U_{\mathcal{P}_2|_C}] \otimes\mathbf{f}[U_{\mathcal{P}_2|_D}]  & {} & \mathbf{f}[U_{\mathcal{P}_1|_A}]\otimes \mathbf{f}[U_{\mathcal{P}_2|_C}] \otimes \mathbf{f}[U_{\mathcal{P}_1|_B}]\otimes \mathbf{f}[U_{\mathcal{P}_2|_D}] \\};
  \path[-stealth]
    (m-1-1) edge node [above] {$\mu_{\mathcal{P}_1,\mathcal{P}_2}$} (m-1-2)
    (m-1-2) edge node [above] {$\Delta_{\mathcal{P}|_{T_1},\mathcal{P}|_{T_2}}$} (m-1-3)
    (m-2-1) edge node [above] {$\cong$} (m-2-3)
    (m-1-1) edge node [fill=white] {$\Delta_{\mathcal{P}_1|_A,\mathcal{P}_1|_B} \otimes \\ \Delta_{\mathcal{P}_2|_C,\mathcal{P}_2|_D}$} (m-2-1)
    (m-2-3) edge node [fill=white] {$\mu_{\mathcal{P}_1|_A,\mathcal{P}_2|_C} \otimes\mu_{\mathcal{P}_1|_B,\mathcal{P}_2|_D}$} (m-1-3);
    \end{tikzpicture}
\end{equation*}
commutes. By functoriality, it is enough to show that the diagram
\begin{equation*}
\begin{tikzpicture}[font=\tiny]
  \matrix (m) [matrix of math nodes,row sep=3em,column sep=1em,minimum width=2em]
  {
     U_{\mathcal{P}_1}\times U_{\mathcal{P}_2} & U_{\mathcal{P}} & U_{\mathcal{P}|_{T_1}} \times U_{\mathcal{P}|_{T_2}} \\
     U_{\mathcal{P}_1|_A}\times U_{\mathcal{P}_1|_B} \times U_{\mathcal{P}_2|_C} \times U_{\mathcal{P}_2|_D}  & {} & U_{\mathcal{P}_1|_A}\times U_{\mathcal{P}_2|_C} \times U_{\mathcal{P}_1|_B} \times U_{\mathcal{P}_2|_D} \\};
  \path[-stealth]
    (m-1-2) edge node [above] {$\pi_{\mathcal{P}_1,\mathcal{P}_2}$} (m-1-1)
    (m-1-3) edge node [above] {$\sigma_{\mathcal{P}|_{T_1},\mathcal{P}|_{T_2}}$} (m-1-2)
    (m-2-3) edge node [below] {$\cong$} (m-2-1)
    (m-2-1) edge node [fill=white] {$\sigma_{\mathcal{P}_1|_A,\mathcal{P}_1|_B} \times\sigma_{\mathcal{P}_2|_C,\mathcal{P}_2|_D}$} (m-1-1)
    (m-1-3) edge node [fill=white] {$\pi_{\mathcal{P}_1|_A,\mathcal{P}_2|_C} \times\pi_{\mathcal{P}_1|_B,\mathcal{P}_2|_D}$} (m-2-3);
    \end{tikzpicture}
\end{equation*}
commutes. Suppose that $(g,h) \in U_{\mathcal{P}|_{T_1}} \times U_{\mathcal{P}|_{T_2}}$; then
\[
        \pi_{\mathcal{P}_1,\mathcal{P}_2}\circ \sigma_{\mathcal{P}|_{T_1},\mathcal{P}|_{T_2}}(g,h) = (x,y),
\]
where
\[
        (x)_{ij} = \left\{\begin{array}{ll}
        g_{ij} & \quad \text{if } i,j \in S_1 \cap T_1, \\
        h_{ij} & \quad \text{if } i,j \in S_1 \cap T_2, \\
        0 & \quad \text{otherwise,}\end{array}\right.
\]
and
\[
        (y)_{ij} = \left\{\begin{array}{ll}
        g_{ij} & \quad \text{if } i,j \in S_2 \cap T_1, \\
        h_{ij} & \quad \text{if } i,j \in S_2 \cap T_2, \\
        0 & \quad \text{otherwise.}\end{array}\right.
\]
At the same time, we have
\[
       (\sigma_{\mathcal{P}_1|_A,\mathcal{P}_1|_B} \times\sigma_{\mathcal{P}_2|_C,\mathcal{P}_2|_D}) \circ \theta \circ (\pi_{\mathcal{P}_1|_A,\mathcal{P}_2|_C} \times\pi_{\mathcal{P}_1|_B,\mathcal{P}_2|_D})(g,h) = (u,v),
\]
where $\theta$ is the isomorphism that switches the middle factors, and $u$ and $v$ are given by
\[
        (u)_{ij} = \left\{\begin{array}{ll}
        g_{ij} & \quad \text{if } i,j \in A, \\
        h_{ij} & \quad \text{if } i,j \in B, \\
        0 & \quad \text{otherwise,}\end{array}\right.
\]
and
\[
        (v)_{ij} = \left\{\begin{array}{ll}
        g_{ij} & \quad \text{if } i,j \in C, \\
        h_{ij} & \quad \text{if } i,j \in D, \\
        0 & \quad \text{otherwise,}\end{array}\right.
\]
thus $(x,y) = (u,v)$.
\end{proof}

\subsection{Submonoids}

For a poset $\mathcal{P}$ on $I$, define
\begin{align*}
    \mathbf{cf}(U_\mathcal{P}) & = \{\text{class functions on }U_\mathcal{P}\}, \\
    \mathbf{scf}(U_\mathcal{P}) & = \bigg\{\begin{array}{l}\text{superclass functions on }U_\mathcal{P}\text{ in the} \\\text{algebra group supercharacter theory}\end{array} \bigg\}, \text{ and} \\
    \mathbf{nnf}(U_\mathcal{P}) & = \bigg\{\begin{array}{l}\text{superclass functions on }U_\mathcal{P}\text{ in the} \\\text{nonnesting supercharacter theory}\end{array} \bigg\}. \\
\end{align*}
Note that
\[
        \mathbf{nnf}(U_\mathcal{P})\subseteq \mathbf{scf}(U_\mathcal{P})\subseteq \mathbf{cf}(U_\mathcal{P})\subseteq\mathbf{f}(U_\mathcal{P}).
\]
Define vector species $\mathbf{cfp}$, $\mathbf{scfp}$, and $\mathbf{nnfp}$ by
\begin{align*}
    \mathbf{cfp}[I] & = \bigoplus_{\substack{ \mathcal{P} \text{ is a}\\ \text{poset on } I}} \mathbf{cf}(U_\mathcal{P}), \\
    \mathbf{scfp}[I] & = \bigoplus_{\substack{ \mathcal{P} \text{ is a}\\ \text{poset on } I}} \mathbf{scf}(U_\mathcal{P}), \text{ and} \\
    \mathbf{nnfp}[I] & = \bigoplus_{\substack{ \mathcal{P} \text{ is a}\\ \text{poset on } I}} \mathbf{nnf}(U_\mathcal{P}).
\end{align*}
These vector species are all subspecies of $\mathbf{fp}$, and in fact define submonoids of $(\mathbf{fp},\mu,\Delta)$.

\begin{proposition} The triple $(\mathbf{cfp},\mu,\Delta)$ defines a Hopf submonoid of $(\mathbf{fp},\mu,\Delta)$.
\end{proposition}

\begin{proof} We only need to show that $\mu$ and $\Delta$ restrict to a product and a coproduct on $\mathbf{cfp}$. This is equivalent to the maps $\sigma_{\mathcal{P}|_S,\mathcal{P}|_T}$ and $\pi_{\mathcal{P},\mathcal{Q}}$ sending conjugate elements to conjugate elements. As $\sigma_{\mathcal{P}|_S,\mathcal{P}|_T}$ and $\pi_{\mathcal{P},\mathcal{Q}}$ are group homomorphisms, this is in fact the case.
\end{proof}

\begin{proposition} The triple $(\mathbf{scfp},\mu,\Delta)$ defines a Hopf submonoid of $(\mathbf{cfp},\mu,\Delta)$.
\end{proposition}

\begin{proof} Once again, we only need show that the maps $\sigma_{\mathcal{P}|_S,\mathcal{P}|_T}$ and $\pi_{\mathcal{P},\mathcal{Q}}$ send elements in the same superclass to elements in the same superclass.

\bigbreak

Let $\mathcal{P}$ and $\mathcal{Q}$ be posets on disjoint sets $S$ and $T$, respectively. Suppose that $g$ and $h$ are in same superclass of $U_{\mathcal{P}\cdot\mathcal{Q}}$; then there exist $u,v \in U_{\mathcal{P}\cdot\mathcal{Q}}$ such that $uf(g)v = f(h)$. We have that
\[
        \pi_{\mathcal{P},\mathcal{Q}}(u)f(\pi_{\mathcal{P},\mathcal{Q}}(g)) \pi_{\mathcal{P},\mathcal{Q}}(v) = f(\pi_{\mathcal{P},\mathcal{Q}}(h)),
\]
hence $\pi_{\mathcal{P},\mathcal{Q}}(g)$ and $\pi_{\mathcal{P},\mathcal{Q}}(h)$ are in the same superclass of $U_\mathcal{P} \times U_\mathcal{Q}$.

\bigbreak

Now let $\mathcal{P}$ be a poset on $I$ and let $I = S \sqcup T$. If $(g_1,g_2),(h_1,h_2) \in U_{\mathcal{P}|_S} \times U_{\mathcal{P}|_T}$ are in the same superclass, then there exist $(u_1,u_2),(v_1,v_2) \in U_{\mathcal{P}|_S} \times U_{\mathcal{P}|_T}$ such that $(u_1,u_2)f((g_1,g_2))(v_1,v_2)= f((h_1,h_2))$, and
\[
     \sigma_{\mathcal{P}|_S,\mathcal{P}|_T}((u_1,u_2)) f(\sigma_{\mathcal{P}|_S,\mathcal{P}|_T}((g_1,g_2))) \sigma_{\mathcal{P}|_S,\mathcal{P}|_T} ((v_1,v_2)) = f(\sigma_{\mathcal{P}|_S,\mathcal{P}|_T}((h_1,h_2))).
\]
\end{proof}

Before showing that $(\mathbf{nnfp},\mu,\Delta)$ is a Hopf submonoid of $(\mathbf{scfp},\mu,\Delta)$, we introduce the idea of the restriction and disjoint union of $(\mathbb{F}_q,\mathcal{P})$-set partitions.

\bigbreak

Let $\mathcal{P}$ be a poset on $I$ and let $S \subseteq I$. If $\eta$ is an $(\mathbb{F}_q,\mathcal{P})$-set partition, define
\[
        \eta|_S = \{i \overset{a}{\frown}j \in \eta \mid i,j \in S\}.
\]
Note that $\eta|_S$ is an $(\mathbb{F}_q,\mathcal{P}|_S)$-set partition, and if $\eta$ is nonnesting then so is $\eta|_S$.

\bigbreak

Suppose that $I = S \sqcup T$; if $\eta$ is an $(\mathbb{F}_q,\mathcal{P}|_S)$-set partition and $\nu$ is an $(\mathbb{F}_q,\mathcal{P}|_T)$-set partition, define
\[
        \eta \sqcup \nu = \{i \overset{a}{\frown}j \mid i \overset{a}{\frown}j \in \eta \text{ or }i \overset{a}{\frown}j \in \nu\}.
\]
Note that $\eta \sqcup \nu$ is an $(\mathbb{F}_q,\mathcal{P})$-set partition; it is not in general the case, however, that if $\eta$ and $\nu$ are both nonnesting then $\eta \sqcup \nu$ will also be nonnesting. To deal with this problem, for any $(\mathbb{F}_q,\mathcal{P})$-set partition $\eta$ we define
\begin{align*}
        \text{sml}(\eta) &= \{i \overset{a}{\frown}j \in \eta \mid \text{ there are no } k\overset{b}{\frown}l \in \eta \text{ with } i\prec_\mathcal{P} k\prec_\mathcal{P}l\prec_\mathcal{P}j\} \text{ and} \\
        \text{big}(\eta) &= \{i \overset{a}{\frown}j \in \eta \mid \text{ there are no } k\overset{b}{\frown}l \in \eta \text{ with } k\prec_\mathcal{P}i\prec_\mathcal{P}j\prec_\mathcal{P}l\}.
\end{align*}
Note that both $\text{sml}(\eta)$ and $\text{big}(\eta)$ are nonnesting. See Example~\ref{exbigsml} for an example in the case that $\mathcal{P}$ is a linear order.

\begin{proposition} The triple $(\mathbf{nnfp},\mu,\Delta)$ defines a Hopf submonoid of $(\mathbf{scfp},\mu,\Delta)$.
\end{proposition}

\begin{proof} Once again, we only need to show that the maps $\sigma_{\mathcal{P}|_S,\mathcal{P}|_T}$ and $\pi_{\mathcal{P},\mathcal{Q}}$ send elements in the same superclass to elements in the same superclass.

\bigbreak

Let $\mathcal{P}$ and $\mathcal{Q}$ be posets on disjoint sets $S$ and $T$, respectively. Suppose that $g$ and $h$ are in same superclass of $U_{\mathcal{P}\cdot\mathcal{Q}}$; then we have $\text{sml}(g) = \text{sml}(h)$, recalling that
\[
        \text{sml}(g) = \bigg\{i\overset{g_{ij}}{\frown}j \;\bigg|\; \begin{array}{l} g_{ij} \neq 0 \text{ and if } (k,l) \neq (i,j) \text{ and}  \\ i \preceq_\mathcal{P} k\prec_\mathcal{P}l \preceq_\mathcal{P} j,\text{ then } g_{kl}=0\end{array}\bigg\}.
\]
Note that if $\pi_{\mathcal{P},\mathcal{Q}}(g) = (g_S,g_T)$, then $\text{sml}(g_S) = \text{sml}(g)|_S$ and $\text{sml}(g_T) = \text{sml}(g)|_T$. It follows that $\pi_{\mathcal{P},\mathcal{Q}}(g)$ and $\pi_{\mathcal{P},\mathcal{Q}}(h)$ are in the same superclass of $U_\mathcal{P} \times U_\mathcal{Q}$.

\bigbreak

Now let $\mathcal{P}$ be a poset on $I$ and let $I = S \sqcup T$. If $(g_1,g_2),(h_1,h_2) \in U_{\mathcal{P}|_S} \times U_{\mathcal{P}|_T}$ are in the same superclass, then $\text{sml}(g_1) = \text{sml}(h_1)$ and $\text{sml}(g_2) = \text{sml}(h_2)$. Observe that
\[
        \text{sml}(\sigma_{\mathcal{P}|_S,\mathcal{P}|_T}((g_1,g_2))) = \text{sml}(\nu_{g_1}\sqcup \nu_{g_2}),
\]
hence $\sigma_{\mathcal{P}|_S,\mathcal{P}|_T}((g_1,g_2))$ and $\sigma_{\mathcal{P}|_S,\mathcal{P}|_T}((h_1,h_2))$ are in the same superclass of $U_{\mathcal{P}}$.
\end{proof}

We mention several other submonoids of interest; in \cite{MR3117506}, Aguiar et al. construct Hopf monoids from the groups of unitriangular matrices with entries in $\mathbb{F}_q$. For a finite set $I$, let $L[I]$ be the set of linear orders on $I$. These are all posets on $I$; let
\begin{align*}
    \mathbf{f}(U)[I] & = \bigoplus_{\mathcal{P}\in L[I]} \mathbf{f}(U_\mathcal{P}), \\
    \mathbf{cf}(U)[I] & = \bigoplus_{\mathcal{P}\in L[I]} \mathbf{cf}(U_\mathcal{P}), \\
    \mathbf{scf}(U)[I] & = \bigoplus_{\mathcal{P}\in L[I]} \mathbf{scf}(U_\mathcal{P}), \text{ and} \\
    \mathbf{nnf}(U)[I] & = \bigoplus_{\mathcal{P}\in L[I]} \mathbf{nnf}(U_\mathcal{P}).
\end{align*}
In \cite{MR3117506}, Aguiar et al.\ show that the triple $(\mathbf{f}(U),\mu,\Delta)$ is a Hopf monoid with submonoids $(\mathbf{cf}(U),\mu,\Delta)$ and $(\mathbf{scf}(U),\mu,\Delta)$.

\begin{corollary} We have the following inclusions of Hopf monoids:
\begin{align*}
        (\mathbf{f}(U),\mu,\Delta) & \subseteq (\mathbf{fp},\mu,\Delta), \\
        (\mathbf{cf}(U),\mu,\Delta) & \subseteq (\mathbf{cfp},\mu,\Delta), \\
        (\mathbf{scf}(U),\mu,\Delta) & \subseteq (\mathbf{scfp},\mu,\Delta), \\
        (\mathbf{nnf}(U),\mu,\Delta) & \subseteq (\mathbf{nnfp},\mu,\Delta),\text{ and} \\
        (\mathbf{nnf}(U),\mu,\Delta) & \subseteq (\mathbf{scf}(U),\mu,\Delta).
\end{align*}
\end{corollary}

\section{Combinatorics of the product and coproduct}\label{seccomb}

The space of superclass functions of $UT_n(\mathbb{F}_q)$ has three standard bases: the superclass basis, the power sum basis, and the supercharacter basis. In \cite{MR3117506}, the product and coproduct of $(\mathbf{scf}(U),\mu,\Delta)$ are described combinatorially in terms of the superclass basis and the power sum basis. There is a combinatorial description of the product in terms of the supercharacter basis provided in \cite[Section~2.3]{MR2880223}, however there is no known combinatorial description of the coproduct in terms of the supercharacter basis. In this section, we provide combinatorial descriptions of the product and coproduct of $(\mathbf{nnfp},\mu,\Delta)$ with respect to the superclass and supercharacter bases and a combinatorial description of the product with respect to the power sum basis. These descriptions restrict to combinatorial descriptions of the product and coproduct of $(\mathbf{nnf}(U),\mu,\Delta)$. With respect to all three bases, the product of nonnesting supercharacters has an identical formula to that for the product of algebra group supercharacters; the coproduct, however, behaves quite differently.

\subsection{The superclass basis}

Let $\mathcal{P}$ be a linear order on $I$; then $U_\mathcal{P} \cong UT_n(\mathbb{F}_q)$ (where $|I|=n$). The supercharacters and superclasses of $U_\mathcal{P}$ in the algebra group supercharacter theory are indexed by the set of $(\mathbb{F}_q,\mathcal{P})$-set partitions, which as in Section~\ref{seclsp} we denote $\Pi(\mathcal{P},q)$. For $\eta \in \Pi(\mathcal{P},q)$, let $K_\eta$ be the superclass of $U_{\mathcal{P}}$ associated to $\eta$; define
\begin{equation}\label{sclbasisutn}
        \kappa_\eta(g) = \left\{\begin{array}{ll}
        1 & \quad \text{if } g \in K_\eta, \\
        0 & \quad \text{otherwise.}\end{array}\right.
\end{equation}
The $\kappa_\eta$ form the \emph{superclass basis} of $\mathbf{scf}(U)$, with respect to which the product and coproduct of $(\mathbf{scf}(U),\mu,\Delta)$ have a nice combinatorial description.

\begin{proposition}[{\cite[Equation 42]{MR3117506}}] Let $\mathcal{P}$ and $\mathcal{Q}$ be linear orders on disjoint sets $S$ and $T$, respectively. If $\eta\in \Pi(\mathcal{P},q)$ and $\nu\in\Pi(\mathcal{Q},q)$, we have
\[
        \mu_{\mathcal{P},\mathcal{Q}}(\kappa_\eta \otimes \kappa_\nu) = \sum_{\substack{\rho \in \Pi(\mathcal{P}\cdot\mathcal{Q},q) \\ \rho|_{S} = \eta , \rho|_{T} = \nu}} \kappa_\rho.
\]
\end{proposition}

\begin{proposition}[{\cite[Equation 43]{MR3117506}}] Let $\mathcal{P}$ be a linear order on $I$ and let $I = S \sqcup T$. If $\eta\in\Pi(\mathcal{P},q)$, we have
\[
        \Delta_{\mathcal{P}|_S,\mathcal{P}|_T}(\kappa_\eta) = \left\{
        \begin{array}{ll}
        \kappa_{\eta|_{S}}\otimes \kappa_{\eta|_{T}} &\quad \text{if } \eta = \eta|_{S}\sqcup\eta|_{T}, \\
        0 & \quad \text{otherwise.}\end{array}\right.
\]
\end{proposition}

Now let $\mathcal{P}$ be any poset on $I$; the supercharacters and superclasses of $U_\mathcal{P}$ in the nonnesting supercharacter theory are indexed by the set of nonnesting $(\mathbb{F}_q,\mathcal{P})$-set partitions, which as in Section~\ref{seclsp} we denote $\text{NN}(\mathcal{P},q)$. For $\eta \in \text{NN}(\mathcal{P},q)$, let $K_\eta$ denote the superclass associated to $\eta$; define
\begin{equation}\label{sclbasisnn}
        \kappa_\eta(g) = \left\{\begin{array}{ll}
        1 & \quad \text{if } g \in K_\eta, \\
        0 & \quad \text{otherwise.}\end{array}\right.
\end{equation}
The $\kappa_\eta$ form the supercharacter basis of $\mathbf{nnfp}$, with respect to which the product and coproduct of $(\mathbf{nnfp},\mu,\Delta)$ have a nice combinatorial description.

\begin{proposition}\label{nnprodscl} Let $\mathcal{P}$ and $\mathcal{Q}$ be posets on disjoint sets $S$ and $T$, respectively. If $\eta\in \text{NN}(\mathcal{P},q)$ and $\nu\in \text{NN}(\mathcal{Q},q)$, then
\[
        \mu_{\mathcal{P},\mathcal{Q}}(\kappa_\eta \otimes \kappa_\nu) = \sum_{\substack{\rho \in \text{NN}(\mathcal{P}\cdot\mathcal{Q},q) \\ \rho|_{S} = \eta , \rho|_{T} = \nu}} \kappa_\rho.
\]
\end{proposition}

\begin{proof} Let $g \in U_{\mathcal{P}\cdot \mathcal{Q}}$, and let $\rho \in \text{NN}(\mathcal{P}\cdot\mathcal{Q},q)$ be such that $g \in K_\rho$. Then we have
\[
        \pi_{\mathcal{P},\mathcal{Q}}(g) \in K_{\rho|_S}\times K_{\rho|_T}.
\]
It follows that
\[
        \mu_{\mathcal{P},\mathcal{Q}}(\kappa_\eta \otimes \kappa_\nu)(g) = \left\{\begin{array}{ll} 1 &\quad\text{if }\rho|_{S} = \eta , \rho|_{T} = \nu, \\
        0 & \quad \text{otherwise.}\end{array}\right.
\]
\end{proof}

\begin{proposition}\label{nncoprodscl} Let $\mathcal{P}$ be a poset on $I$ and let $I = S \sqcup T$. If $\eta\in \text{NN}(\mathcal{P},q)$, then
\[
        \Delta_{\mathcal{P}|_S,\mathcal{P}|_T}(\kappa_\eta) = \sum_{\substack{\nu \in \text{NN}(\mathcal{P}|_S,q) \\ \rho \in \text{NN}(\mathcal{P}|_T,q) \\ \textup{sml}(\nu\sqcup\rho) = \eta}} \kappa_\nu \otimes \kappa_\rho.
\]
\end{proposition}

\begin{proof} Let $(g,h) \in U_{\mathcal{P}|_S}\times U_{\mathcal{P}|_T}$, with $g \in K_\nu$ and $h \in K_\rho$; then
\[
        \sigma_{\mathcal{P}|_S,\mathcal{P}|_T}(g,h) \in K_{\text{sml}(\nu\sqcup\rho)}.
\]
It follows that
\[
        \varphi(\Delta_{\mathcal{P}|_S,\mathcal{P}|_T}(\kappa_\eta))((g,h))
        = \left\{\begin{array}{ll}
        1 &\quad \text{if } \eta = \text{sml}(\nu\sqcup\rho), \\
        0 & \quad \text{otherwise,}\end{array}\right.
\]
where once again $\varphi$ is the canonical isomorphism from $\mathbf{f}(U_{\mathcal{P}|_S}) \otimes \mathbf{f}(U_{\mathcal{P}|_T})$ to $\mathbf{f}(U_{\mathcal{P}|_S}\times U_{\mathcal{P}|_T})$.
\end{proof}

\begin{remark} It is worth noting that although the formula for the product in $(\mathbf{nnfp},\mu,\Delta)$ with respect to the superclass basis is identical to that of $(\mathbf{scf}((U),\mu,\Delta)$, the same is not true of the coproduct. We do, however, still have that $\Delta_{\mathcal{P}|_S,\mathcal{P}|_T}(\kappa_\eta) = 0$ unless $\eta = \eta|_S \sqcup\eta|_T$.
\end{remark}

\subsection{The power sum basis}\label{secpsb}

Given two $(\mathbb{F}_q,\mathcal{P})$-set partitions $\eta$ and $\nu$, we say that $\eta \subseteq \nu$ if the edge set of $\eta$ is contained in the edge set of $\nu$. Let $\mathcal{P}$ be a linear order on $I$; for an $(\mathbb{F}_q,\mathcal{P})$-set partition $\eta$, define
\begin{equation}\label{psbasisutn}
        p_\eta = \sum_{\substack{\nu \in \Pi(\mathcal{P},q) \\ \eta \subseteq \nu}} \kappa_\nu,
\end{equation}
where $\kappa_\nu$ is as in Equation~\ref{sclbasisutn}. The functions $p_\eta$ form the \emph{power sum basis} of $\mathbf{scf}(U)$, with respect to which the product and coproduct of $(\mathbf{scf}(U),\mu,\Delta)$ have a nice combinatorial description.

\begin{proposition}[{\cite[Proposition 14]{MR3117506}}] Let $\mathcal{P}$ and $\mathcal{Q}$ be linear orders on disjoint sets $S$ and $T$, respectively. If $\eta\in \Pi(\mathcal{P},q)$ and $\nu\in\Pi(\mathcal{Q},q)$, we have
\[
        \mu_{\mathcal{P},\mathcal{Q}}(p_\eta \otimes p_\nu) = p_{\eta\sqcup\nu}.
\]
\end{proposition}

\begin{proposition}[{\cite[Proposition 15]{MR3117506}}] Let $\mathcal{P}$ be a linear order on $I$ and let $I = S \sqcup T$. If $\eta\in\Pi(\mathcal{P},q)$, we have
\[
        \Delta_{\mathcal{P}|_S,\mathcal{P}|_T}(p_\eta) = \left\{
        \begin{array}{ll}
        p_{\eta|_{S}}\otimes p_{\eta|_{T}} &\quad \text{if } \eta = \eta|_{S}\sqcup\eta|_{T}, \\
        0 & \quad \text{otherwise.}\end{array}\right.
\]
\end{proposition}

Now let $\mathcal{P}$ be any poset on $I$. For a nonnesting $(\mathbb{F}_q,\mathcal{P})$-set partition $\eta$, define
\begin{equation}\label{psbasisup}
        p_\eta = \sum_{\substack{\nu \in \text{NN}(\mathcal{P},q) \\ \eta \subseteq \nu}} \kappa_\nu,
\end{equation}
where $\kappa_\nu$ is as in Equation~\ref{sclbasisnn}. The functions $p_\eta$ form the power sum basis of $\mathbf{nnfp}$, with respect to which the product of $(\mathbf{nnfp},\mu,\Delta)$ has a nice combinatorial description.

\begin{proposition} Let $\mathcal{P}$ and $\mathcal{Q}$ be posets on disjoint sets $S$ and $T$, respectively. If $\eta\in \text{NN}(\mathcal{P},q)$ and $\nu\in \text{NN}(\mathcal{Q},q)$, we have
\[
        \mu_{\mathcal{P},\mathcal{Q}}(p_\eta \otimes p_\nu) = p_{\eta\sqcup\nu}.
\]
\end{proposition}

\begin{proof} We have that
\begin{align*}
        \mu_{\mathcal{P},\mathcal{Q}}(p_\eta \otimes p_\nu)
        &=  \sum_{\substack{\alpha \in \text{NN}(\mathcal{P},q) \\ \eta \subseteq \alpha}}
            \sum_{\substack{\beta \in \text{NN}(\mathcal{Q},q) \\ \nu \subseteq \beta}}
            \mu_{\mathcal{P},\mathcal{Q}}(\kappa_\alpha \otimes \kappa_\beta) \\
        & = \sum_{\substack{\rho \in \text{NN}(\mathcal{P}\cdot\mathcal{Q}) \\ \eta \subseteq \rho|_S, \nu\subseteq \rho|_T}} \kappa_\rho
\end{align*}
by Proposition~\ref{nnprodscl}. At the same time, an element $\rho \in \text{NN}(\mathcal{P}\cdot \mathcal{Q},q)$ has the property that $\eta \subseteq \rho|_S$ and $\nu\subseteq \rho|_T$ if and only if $\eta \sqcup \nu \subseteq \rho$.
\end{proof}

Note that $\eta \sqcup\nu$ is a nonnesting $(\mathbb{F}_q,\mathcal{P}\cdot\mathcal{Q})$-set partition even though in general disjoint unions of nonnesting set partitions need not be nonnesting.

\bigbreak

We remark that the coproduct on the power sum basis does not behave as one might hope. The following example illustrates where things can get complicated.

\begin{example} If we put the usual linear order on $\{1,2,3,4\}$, then we have
\begin{align*}
        \Delta_{\{1,4\},\{2,3\}}\bigg(p\begin{tikzpicture}[baseline={([yshift=1ex]current bounding box.center)}]
	\fill (0,0) circle (.1);
    \fill (.5,0) circle (.1);
    \fill (1,0) circle (.1);
    \fill (1.5,0) circle (.1);
    \draw (.5,0) to [out=45, in=135] node[above] {$a$} (1,0);
    \end{tikzpicture}\bigg)
    &= \Delta_{\{1,4\},\{2,3\}}\bigg(\kappa\begin{tikzpicture}[baseline={([yshift=1ex]current bounding box.center)}]
	\fill (0,0) circle (.1);
    \fill (.5,0) circle (.1);
    \fill (1,0) circle (.1);
    \fill (1.5,0) circle (.1);
    \draw (.5,0) to [out=45, in=135] node[above] {$a$} (1,0);
    \end{tikzpicture}\bigg) \\
    & = \kappa\begin{tikzpicture}[baseline={([yshift=1ex]current bounding box.center)}]
	\fill (0,0) circle (.1);
    \fill (.5,0) circle (.1);
    \end{tikzpicture} \otimes
    \kappa\begin{tikzpicture}[baseline={([yshift=1ex]current bounding box.center)}]
	\fill (0,0) circle (.1);
    \fill (.5,0) circle (.1);
    \draw (0,0) to [out=45, in=135] node[above] {$a$} (.5,0);
    \end{tikzpicture} +
    \sum_{b \in \mathbb{F}_q^\times}
    \kappa\begin{tikzpicture}[baseline={([yshift=1ex]current bounding box.center)}]
	\fill (0,0) circle (.1);
    \fill (.5,0) circle (.1);
    \draw (0,0) to [out=45, in=135] node[above] {$b$} (.5,0);
    \end{tikzpicture} \otimes
    \kappa\begin{tikzpicture}[baseline={([yshift=1ex]current bounding box.center)}]
	\fill (0,0) circle (.1);
    \fill (.5,0) circle (.1);
    \draw (0,0) to [out=45, in=135] node[above] {$a$} (.5,0);
    \end{tikzpicture} \\
    & = p\begin{tikzpicture}[baseline={([yshift=1ex]current bounding box.center)}]
	\fill (0,0) circle (.1);
    \fill (.5,0) circle (.1);
    \end{tikzpicture} \otimes
    p\begin{tikzpicture}[baseline={([yshift=1ex]current bounding box.center)}]
	\fill (0,0) circle (.1);
    \fill (.5,0) circle (.1);
        \draw (0,0) to [out=45, in=135] node[above] {$a$} (.5,0);
    \end{tikzpicture},
\end{align*}
which is the behavior we might hope for. On the other hand, we have
\begin{align*}
        \Delta_{\{1,4\},\{2,3\}}\bigg(p\begin{tikzpicture}[baseline={([yshift=1ex]current bounding box.center)}]
	\fill (0,0) circle (.1);
    \fill (.5,0) circle (.1);
    \fill (1,0) circle (.1);
    \fill (1.5,0) circle (.1);
    \draw (0,0) to [out=45, in=135] node[above] {$a$} (1.5,0);
    \end{tikzpicture}\bigg)
    &= \Delta_{\{1,4\},\{2,3\}}\bigg(\kappa\begin{tikzpicture}[baseline={([yshift=1ex]current bounding box.center)}]
	\fill (0,0) circle (.1);
    \fill (.5,0) circle (.1);
    \fill (1,0) circle (.1);
    \fill (1.5,0) circle (.1);
    \draw (0,0) to [out=45, in=135] node[above] {$a$} (1.5,0);
    \end{tikzpicture}\bigg) \\
    & = \kappa\begin{tikzpicture}[baseline={([yshift=1ex]current bounding box.center)}]
	\fill (0,0) circle (.1);
    \fill (.5,0) circle (.1);
    \draw (0,0) to [out=45, in=135] node[above] {$a$} (.5,0);
    \end{tikzpicture} \otimes
    \kappa\begin{tikzpicture}[baseline={([yshift=1ex]current bounding box.center)}]
	\fill (0,0) circle (.1);
    \fill (.5,0) circle (.1);
    \end{tikzpicture} \\
    & = p\begin{tikzpicture}[baseline={([yshift=1ex]current bounding box.center)}]
	\fill (0,0) circle (.1);
    \fill (.5,0) circle (.1);
    \draw (0,0) to [out=45, in=135] node[above] {$a$} (.5,0);
    \end{tikzpicture} \otimes
    p\begin{tikzpicture}[baseline={([yshift=1ex]current bounding box.center)}]
	\fill (0,0) circle (.1);
    \fill (.5,0) circle (.1);
    \end{tikzpicture} - \sum_{b \in \mathbb{F}_q^\times}
    p\begin{tikzpicture}[baseline={([yshift=1ex]current bounding box.center)}]
	\fill (0,0) circle (.1);
    \fill (.5,0) circle (.1);
    \draw (0,0) to [out=45, in=135] node[above] {$a$} (.5,0);
    \end{tikzpicture} \otimes
    p\begin{tikzpicture}[baseline={([yshift=1ex]current bounding box.center)}]
	\fill (0,0) circle (.1);
    \fill (.5,0) circle (.1);
    \draw (0,0) to [out=45, in=135] node[above] {$b$} (.5,0);
    \end{tikzpicture},
\end{align*}
and the coproduct on the power sum basis displays some asymmetry.
\end{example}

\subsection{The supercharacter basis}

If $\mathcal{P}$ is a linear order on $I$, the supercharacters of $U_\mathcal{P}$ in the algebra group supercharacter theory are indexed by the $(\mathbb{F}_q,\mathcal{P})$-set partitions. The set $\{\chi_\eta \mid \eta \in \Pi(\mathcal{P},q)\}$ is the \emph{supercharacter basis} of $\mathbf{scf}(U_\mathcal{P})$. The product of $(\mathbf{scf}(U),\mu,\Delta)$ has a nice combinatorial description with respect to this basis.

\begin{proposition}[{\cite[Section~2.3]{MR2880223}}] Let $\mathcal{P}$ and $\mathcal{Q}$ be linear orders on disjoint sets $S$ and $T$, respectively. If $\eta\in \Pi(\mathcal{P},q)$ and $\nu\in\Pi(\mathcal{Q},q)$, we have
\[
        \mu_{\mathcal{P},\mathcal{Q}}(\chi_\eta \otimes \chi_\nu) = \chi_{\eta\sqcup\nu}.
\]
\end{proposition}

If $\mathcal{P}$ is any poset on $I$, the supercharacters of $U_\mathcal{P}$ in the nonnesting supercharacter theory are indexed by the nonnesting $(\mathbb{F}_q,\mathcal{P})$-set partitions. The set $\{\chi_\eta \mid \eta \in \text{NN}(\mathcal{P},q)\}$ is the supercharacter basis of $\mathbf{nnfp}(U_\mathcal{P})$, and with respect to this basis the product of $(\mathbf{nnfp},\mu,\Delta)$ has a nice combinatorial description.

\begin{proposition} Let $\mathcal{P}$ and $\mathcal{Q}$ be posets on disjoint sets $S$ and $T$, respectively. If $\eta\in \text{NN}(\mathcal{P},q)$ and $\nu\in \text{NN}(\mathcal{Q},q)$, we have
\[
        \mu_{\mathcal{P},\mathcal{Q}}(\chi_\eta \otimes \chi_\nu) = \chi_{\eta\sqcup\nu}.
\]
\end{proposition}

\begin{proof} Let $g \in U_{\mathcal{P}\cdot\mathcal{Q}}$ be in the superclass $K_\rho$. Then we have
\[
        \mu_{\mathcal{P},\mathcal{Q}}(\chi_\eta \otimes \chi_\nu)(g) = (\chi_\eta,\chi_\nu)(\pi_{\mathcal{P},\mathcal{Q}}(g)).
\]
Note that $\pi_{\mathcal{P},\mathcal{Q}}(g) \in K_{\rho|_S} \times K_{\rho|_T}$, hence for $h \in K_{\rho|_S}$ and $k \in K_{\rho|_T}$,
\[
        \mu_{\mathcal{P},\mathcal{Q}}(\chi_\eta \otimes \chi_\nu)(g) = \chi_\eta (h) \chi_\nu (k).
\]
By Proposition~\ref{anotherprop},
\begin{align*}
        &\mu_{\mathcal{P},\mathcal{Q}}(\chi_\eta \otimes \chi_\nu)(g)\\ 
        &\hspace{.2in} = \left\{\begin{array}{ll}
        \chi_\eta(1)\chi_\nu(1)\displaystyle\prod_{\substack{i\overset{a}{\frown}j \in \eta\sqcup\nu \\
        i \overset{b}{\frown}{j} \in \rho}} \theta(ab)   & \begin{array}{l} \text{if there are no } i \overset{a}{\frown}j \in \eta \text{ and } k\overset{b}{\frown}l \in \rho|_S \\ \text{with } (i,j) \neq (k,l) \text{ and }
        i \preceq_\mathcal{P} k\prec_\mathcal{P}l\preceq_\mathcal{P} j, \\ \text{and no } i \overset{a}{\frown}j \in \nu \text{ and } k\overset{b}{\frown}l \in \rho|_T \\ \text{with } (i,j) \neq (k,l) \text{ and }
        i \preceq_\mathcal{P} k\prec_\mathcal{P}l\preceq_\mathcal{P} j,\end{array} \\
        0   & \:\;\text{otherwise,}\end{array}\right. \\
        &\hspace{.2in}  = \left\{\begin{array}{ll}
        \chi_\eta(1)\chi_\nu(1)\displaystyle\prod_{\substack{i\overset{a}{\frown}j \in \eta\sqcup\nu \\
        i \overset{b}{\frown}{j} \in \rho}} \theta(ab)   & \begin{array}{l} \text{if there are no } i \overset{a}{\frown}j \in \eta\sqcup\nu \text{ and } k\overset{b}{\frown}l \in \rho \\ \text{with } (i,j) \neq (k,l) \text{ and }
        i \preceq_\mathcal{P} k\prec_\mathcal{P}l\preceq_\mathcal{P} j,\end{array} \\
        0   & \:\;\text{otherwise,}\end{array}\right. \\
        &\hspace{.2in}  = \frac{\chi_\eta(1)\chi_\nu(1)}{\chi_{\eta\sqcup\nu}(1)}\chi_{\eta\sqcup\nu}(g).
\end{align*}
By Corollary~\ref{nnpdim}, $\chi_\eta(1)\chi_\nu(1) = \chi_{\eta\sqcup\nu}(1)$.
\end{proof}

There is no known combinatorial description of the coproduct of $(\mathbf{scf}(U),\mu,\Delta)$ in terms of the supercharacter basis. This is related to the fact that the restriction of an algebra group supercharacter of $UT_n(\mathbb{F}_q)$ to a smaller unitriangular matrix group is difficult to decompose into supercharacters. Some rules for restricting supercharacters can be found in \cite{MR2592079,MR2482091}. In order to describe the coproduct of $(\mathbf{nnfp},\mu,\Delta)$ in terms of the supercharacter basis, we need the following definition.

\begin{definition}
Let $I$ be a finite set and let $\mathcal{P}$ be a poset on $I$. Given a subset $S \subseteq I$ and a nonnesting $(\mathbb{F}_q,\mathcal{P})$-set partition $\eta$, define the \emph{projection of} $\eta$ \emph{from} $I$ \emph{to} $S$, denoted $\text{proj}_S^I(\eta)$, to be the set of nonnesting $(\mathbb{F}_q,\mathcal{P}|_S)$-set partitions $\nu$ such that
\begin{enumerate}
\item if $i \overset{a}{\frown}j \in \eta$ and $i,j \in S$, then $i \overset{a}{\frown}j \in \nu$; and
\item if $i \overset{a}{\frown}j \in \nu$, then there exist $k,l \in I$ with $k\overset{b}{\frown}l \in \eta$ and $k \preceq_\mathcal{P} i \prec_\mathcal{P} j \preceq_\mathcal{P} l$.
\end{enumerate}
\end{definition}
In other words, $\text{proj}_S^I(\eta)$ consists of the nonnesting $(\mathbb{F}_q,\mathcal{P}|_S)$-set partitions that contain all of the arcs of $\eta|_S$ and perhaps some other arcs that were prohibited by arcs of $\eta$.

\begin{example}
Let $I = \{1,2,3,4,5\}$, $\mathcal{P}$ be the usual linear order, $S = \{1,2,3,4\}$, and
\[
        \eta = \begin{tikzpicture}
	\fill (0,0) circle (.1);
    \fill (.5,0) circle (.1);
    \fill (1,0) circle (.1);
    \fill (1.5,0) circle (.1);
    \fill (2,0) circle (.1);
    \draw (0,0) to [out=45, in=135] node[above] {$a$} (.5,0);
    \draw (1,0) to [out=45, in=135] node[above] {$b$} (2,0);
    \end{tikzpicture};
\]
then
\[
    \text{proj}_S^I(\eta) = \{\begin{tikzpicture}
	\fill (0,0) circle (.1);
    \fill (.5,0) circle (.1);
    \fill (1,0) circle (.1);
    \fill (1.5,0) circle (.1);
    \draw (0,0) to [out=45, in=135] node[above] {$a$} (.5,0);
    \end{tikzpicture}\}
    \cup \{\begin{tikzpicture}
	\fill (0,0) circle (.1);
    \fill (.5,0) circle (.1);
    \fill (1,0) circle (.1);
    \fill (1.5,0) circle (.1);
    \draw (0,0) to [out=45, in=135] node[above] {$a$} (.5,0);
    \draw (1,0) to [out=45, in=135] node[above] {$c$} (1.5,0);
    \end{tikzpicture} \mid c \in \mathbb{F}_q^\times\}.
\]
\end{example}

Define
\[
        U_{\eta,S} = \{u \in U_{\mathcal{P}|_S} \mid  \sigma_{\mathcal{P}|_S,\mathcal{P}|_T}(u,1) \in U_\eta\},
\]
where $U_\eta$ is as in Lemma~\ref{ueta}. Note that $U_{\eta,S} \subseteq U_{\eta|_S}$, and in general the two are not the same group. Let $\frak{u}_{\eta,S} = f(U_{\eta,S})$, and note that this is a subalgebra of $\frak{u}_{\mathcal{P}|_S}$.

\bigbreak

We can describe the coproduct of $(\mathbf{nnfp},\mu,\Delta)$ using the projection of nonnesting $(\mathbb{F}_q,\mathcal{P})$-set partitions.

\begin{proposition} Let $I=S \sqcup T$ be a finite set and let $\mathcal{P}$ be a poset on $I$. If $\eta$ is a nonnesting $(\mathbb{F}_q,\mathcal{P})$-set partition, then
\[
        \Delta_{\mathcal{P}|_S,\mathcal{P}|_T}(\chi_\eta)
        = \frac{|U_\mathcal{P}:U_\eta|}{|U_{\mathcal{P}|_S}:U_{\eta,S}| |U_{\mathcal{P}|_T}:U_{\eta,T}|}\left(\sum_{\nu \in \textup{proj}_S^I(\eta)}\chi_\nu\right) \otimes \left(\sum_{\rho \in \textup{proj}_T^I(\eta)}\chi_\rho\right).
\]
\end{proposition}

\begin{proof} Let $(g,h) \in U_{\mathcal{P}|_S}\times U_{\mathcal{P}|_T}$, with $g \in K_\nu$ and $h \in K_\rho$; then
\[
        \sigma_{\mathcal{P}|_S,\mathcal{P}|_T}(g,h) \in K_{\text{sml}(\nu\sqcup\rho)}.
\]
It follows that if $k \in K_{\text{sml}(\nu\sqcup\rho)}$, we have
\[
        \varphi(\Delta_{\mathcal{P}|_S,\mathcal{P}|_T}(\chi_\eta))((g,h)) = \chi_\eta(k).
\]
By Proposition~\ref{anotherprop},
\begin{align*}
        &\varphi(\Delta_{\mathcal{P}|_S,\mathcal{P}|_T}(\chi_\eta))((g,h)) \\
        &\hspace{.2in} = \left\{\begin{array}{ll}
        |U_\mathcal{P}:U_\eta|\displaystyle\prod_{\substack{i\overset{a}{\frown}j \in \eta \\
        i \overset{b}{\frown}{j} \in \text{sml}(\nu\sqcup\rho)}} \theta(ab) & \quad \begin{array}{l} \text{if there are no } i \overset{a}{\frown}j \in \eta \text{ and} \\ k\overset{b}{\frown}l \in \text{sml}(\nu\sqcup\rho) \text{ with} \\ (i,j) \neq (k,l) \text{ and }
        i \preceq_\mathcal{P} k\prec_\mathcal{P}l\preceq_\mathcal{P} j,\end{array} \\
        0 & \quad \text{otherwise.}\end{array}\right.
\end{align*}
The condition ``there are no $i \overset{a}{\frown}j \in \eta$ and $k\overset{b}{\frown}l \in \text{sml}(\nu\sqcup\rho)$ with $(i,j) \neq (k,l)$ and $i \preceq_\mathcal{P} k\prec_\mathcal{P}l\preceq_\mathcal{P} j$'' is equivalent to the condition ``there are no $i \overset{a}{\frown}j \in \eta$ and $k\overset{b}{\frown}l \in \nu\sqcup\rho$ with $(i,j) \neq (k,l)$ and $i \preceq_\mathcal{P} k\prec_\mathcal{P}l\preceq_\mathcal{P} j$.'' This is a consequence of the definition of $\text{sml}(\nu\sqcup\rho)$. Furthermore, if there are no $i \overset{a}{\frown}j \in \eta$ and $k\overset{b}{\frown}l \in \nu\sqcup\rho$ with $(i,j) \neq (k,l)$ and $i \preceq_\mathcal{P} k\prec_\mathcal{P}l\preceq_\mathcal{P} j$, then $i \overset{a}{\frown}j \in \eta$ and $i \overset{b}{\frown}j \in \text{sml}(\nu\sqcup\rho)$ if and only if $i \overset{a}{\frown}j \in \eta$ and $i \overset{b}{\frown}j \in \nu\sqcup\rho$. We now have that
\begin{align*}
        &\varphi(\Delta_{\mathcal{P}|_S,\mathcal{P}|_T}(\chi_\eta))((g,h)) \\
        &\hspace{.2in} = \left\{\begin{array}{ll}
        |U_\mathcal{P}:U_\eta|\displaystyle\prod_{\substack{i\overset{a}{\frown}j \in \eta \\
        i \overset{b}{\frown}{j} \in \nu\sqcup\rho}} \theta(ab) & \quad \begin{array}{l} \text{if there are no } i \overset{a}{\frown}j \in \eta \text{ and} \\ k\overset{b}{\frown}l \in \nu\sqcup\rho \text{ with} \\ (i,j) \neq (k,l) \text{ and }
        i \preceq_\mathcal{P} k\prec_\mathcal{P}l\preceq_\mathcal{P} j,\end{array} \\
        0 & \quad \text{otherwise.}\end{array}\right.
\end{align*}
Let $\lambda_S \in \frak{u}_{\eta,S}^*$ be the functional defined by
\[
        \lambda_S(x) = \sum_{i \overset{a_{ij}}{\frown} j \in \eta|_S} a_{ij}x_{ij}
\]
for any $x \in \frak{u}_{\eta,S}$. Note that $\theta \circ \lambda_S \circ f$ is a $U_{\mathcal{P}|_S}$-invariant linear character of $U_{\eta,S}$, hence
\begin{align*}
        &\text{Ind}_{U_{\eta,S}}^{U_{\mathcal{P}|_S}}(\theta \circ \lambda_S \circ f)(g) \\
        &\hspace{.2in}  =
        \left\{\begin{array}{ll}
        |U_{\mathcal{P}|_S}:U_{\eta,S}|\displaystyle\prod_{\substack{i\overset{a}{\frown}j \in \eta \\
        i \overset{b}{\frown}{j} \in \nu}} \theta(ab) & \quad \begin{array}{l} \text{if there are no } i \overset{a}{\frown}j \in \eta \text{ and } k\overset{b}{\frown}l \in \nu \\ \text{with } (i,j) \neq (k,l) \text{ and }
        i \preceq_\mathcal{P} k\prec_\mathcal{P}l\preceq_\mathcal{P} j,\end{array} \\
        0 & \quad \text{otherwise.}\end{array}\right.
\end{align*}
Similarly define $\lambda_T \in \frak{u}_{\eta,T}^*$; we have that
\begin{align*}
        &\varphi(\Delta_{\mathcal{P}|_S,\mathcal{P}|_T}(\chi_\eta))((g,h)) \\
        &\hspace{.2in} = \frac{|U_\mathcal{P}:U_\eta|}{|U_{\mathcal{P}|_S}:U_{\eta,S}| |U_{\mathcal{P}|_T}:U_{\eta,T}|}\text{Ind}_{U_{\eta,S}}^{U_{\mathcal{P}|_S}}(\theta \circ \lambda_S \circ f)(g)
        \text{Ind}_{U_{\eta,T}}^{U_{\mathcal{P}|_T}}(\theta \circ \lambda_T \circ f)(h).
\end{align*}
To complete the proof, we need to show that
\[
        \text{Ind}_{U_{\eta,S}}^{U_{\mathcal{P}|_S}}(\theta \circ \lambda_S \circ f) = \sum_{\nu \in \textup{proj}_S^I(\eta)}\chi_\nu.
\]
By \cite[Lemma~4.5]{andrews1}, we have
\[
        \text{Ind}_{U_{\eta,S}}^{U_{\mathcal{P}|_S}}(\theta \circ \lambda_S \circ f) =
        \sum_{\substack{\mu \in \frak{u}_{\mathcal{P}|_S}^* \\ \mu|_{\frak{u}_{\eta,S}} = \lambda_S}} \theta \circ \mu \circ f.
\]
At the same time, a functional $\mu \in \frak{u}_{\mathcal{P}|_S}^*$ has $\text{big}(\mu) \in \text{proj}_S^I(\eta)$ if and only if $\mu|_{\frak{u}_{\eta,S}} = \lambda_S$. The claim follows from the definition of the supercharacter $\chi_\nu$.
\end{proof}

\subsection{Freeness}\label{subsecfree}

In this section we show that the Hopf monoids $(\mathbf{nnfp},\mu,\Delta)$ and $(\mathbf{nnf}(U),\mu,\Delta)$ are free. This is analogous to the result for $(\mathbf{scf}(U),\mu,\Delta)$ presented in \cite[Section~5]{MR3117506}.

\bigbreak

Let $\mathcal{P}$ be a poset on $I$; an $(\mathbb{F}_q,\mathcal{P})$-set partition $\eta$ is called \emph{atomic} (see  \cite[Section~5]{MR3117506}) if there are no nonempty sets $S$ and $T$ with
\begin{enumerate}
\item $I = S \sqcup T$,
\item $\mathcal{P} = \mathcal{P}|_S \cdot \mathcal{P}|_T$, and
\item $\eta = \eta|_S \sqcup \eta|_T$.
\end{enumerate}
We mention that if $\mathcal{P}$ is a linear order on $I$, this definition reduces to that given in \cite[Section~5]{MR3117506}.

\bigbreak

For $\mathcal{P}$ a poset on $I$, let $\text{D}(\mathcal{P})$ be the set of atomic $(\mathbb{F}_q,\mathcal{P})$-set partitions, and let $\mathbf{d}(\mathcal{P})$ be the vector space with basis $\text{D}(\mathcal{P})$. Define a species $\mathbf{d}$ by
\[
        \mathbf{d}[I] = \bigoplus_{\substack{\mathcal{P}\text{ is a linear} \\\text{order on }I}}\mathbf{d}(\mathcal{P}).
\]
\begin{proposition}[{\cite[Corollary~18]{MR3117506}}] There exists an isomorphism of Hopf monoids $\mathbf{scf}(U) \cong \mathcal{T}(\mathbf{d})$, where $\mathcal{T}(\mathbf{d})$ is the free Hopf monoid on $\mathbf{d}$.
\end{proposition}

In order to prove analogous results for $(\mathbf{nnfp},\mu,\Delta)$ and $(\mathbf{nnf}(U),\mu,\Delta)$, define
\[
        \text{NND}(\mathcal{P}) = \{\text{Atomic nonnesting }(\mathbb{F}_q,\mathcal{P})\text{-set partitions}\},
\]
and let $\mathbf{nnd}(\mathcal{P})$ be the vector space with basis $\text{NND}(\mathcal{P})$. Define the species $\mathbf{nnd}$ and $\mathbf{nndp}$ by
\begin{align*}
        \mathbf{nnd}[I] & = \bigoplus_{\substack{\mathcal{P}\text{ is a linear} \\\text{order on }I}}\mathbf{nnd}(\mathcal{P}) \quad \text{and} \\
        \mathbf{nndp}[I] & = \bigoplus_{\substack{\mathcal{P}\text{ is a } \\\text{poset on }I}}\mathbf{nnd}(\mathcal{P})
\end{align*}

\begin{proposition}\label{propnnfree} There exist isomorphisms of Hopf monoids
\[
        \mathbf{nnf}(U) \cong \mathcal{T}(\mathbf{nnd}) \quad \text{and} \quad \mathbf{nnfp} \cong \mathcal{T}(\mathbf{nndp}).
\]
\end{proposition}

\begin{proof} We will only prove that there exists an isomorphism $\mathbf{nnfp} \cong \mathcal{T}(\mathbf{nndp})$, the other result is analogous.

\bigbreak

Define a morphism of species
\begin{align*}
       \widetilde{\varphi}: \mathbf{nndp} & \to \mathbf{nnfp} \\
        \eta & \mapsto \lambda_\eta,
\end{align*}
where $\lambda_\eta$ is the power sum basis element of $\mathbf{nnfp}$ defined in Section~\ref{secpsb}. As $\mathcal{T}(\mathbf{nndp})$ is free, by \cite[Theorem 11.4]{MR2724388} this morphism uniquely extends to a morphism of monoids
\[
        \varphi:\mathcal{T}(\mathbf{nndp})  \to \mathbf{nnfp}.
\]
By Proposition~\ref{propfree}, it suffices to show that $\varphi$ is a bijection.

\bigbreak

Given a poset $\mathcal{P}$ on $I$ and a nonnesting $(\mathbb{F}_q,\mathcal{P})$-set partition $\eta$, we can uniquely write
\[
        \eta = \eta|_{S_1} \sqcup \eta|_{S_2} \sqcup \hdots \sqcup \eta|_{S_k}
\]
such that $\mathcal{P} = \mathcal{P}|_{S_1} \cdot \mathcal{P}|_{S_2} \cdot \hdots \cdot \mathcal{P}|_{S_k}$ and each $\eta|_{S_i}$ is atomic. In other words,
\[
        \lambda_\eta = \mu_{S_1,S_2,\hdots,S_k}(\lambda_{\eta|_{S_1}} \otimes \lambda_{\eta|_{S_2}}\otimes \hdots \otimes \lambda_{\eta|_{S_k}}),
\]
and this decomposition is unique. It follows that $\varphi$ is an isomorphism, and $\mathbf{nnfp} \cong \mathcal{T}(\mathbf{nndp})$ as Hopf monoids.
\end{proof}

\begin{remark} In the proof of Proposition~\ref{propnnfree} we could have used the supercharacter basis instead of the power sum basis as the combinatorics of the product is identical with respect to these bases.
\end{remark}

\section{Further directions}

In \cite{MR2880223}, Aguiar et al. show that for $q=2$, the Hopf algebra of superclass functions of $UT_n(\mathbb{F}_q)$ is isomorphic to $\text{NCSym}(X)$, the Hopf algebra of symmetric functions in noncommuting variables. Bergeron--Thiem produce a coarser supercharacter theory of $UT_n(\mathbb{F}_q)$ in \cite{MR3078055}, from which they construct a Hopf algebra that is isomorphic to $\text{NCSym}(X)$ for any $q$. The coarsening is accomplished by combining supercharacters and superclasses that are indexed by $\mathbb{F}_q$-set partitions of the same shape (but with different labels). This method also produces a coarser version of the nonnesting supercharacter theory and a corresponding Hopf algebra; applying the isomorphism from \cite{MR2880223}, we obtain a subalgebra of $\text{NCSym}(X)$. It would be interesting to know if there is a nice description of this subalgebra with respect to one of the usual bases of $\text{NCSym}(X)$.

\bigbreak

The algebra group supercharacter theory has been generalized to the unipotent orthogonal, symplectic, and unitary groups by Andr\'e--Neto in \cite{MR2457229,MR2264135,MR2537684} and the author in \cite{andrews1}. There are several possible ways to define a nonnesting supercharacter theory in these types, and the author plans to explore this in a later paper.

\section{Acknowledgements}

I would like to thank Nat Thiem for his numerous helpful suggestions and insights.

\bibliography{bibfile}
	\bibliographystyle{plain}

\end{document}